\documentclass[a4paper, 11pt]{article}
\usepackage{amsmath}
\usepackage{amssymb}
\usepackage{amsthm}
\usepackage{amsfonts}
\usepackage{latexsym}
\usepackage{mathtools}
\usepackage{empheq}
\usepackage{comment}
\usepackage{cases}
\usepackage{float}
\usepackage{color}
\usepackage[top=26truemm, bottom=25truemm, left=25.5truemm, right=25.5truemm]{geometry}
\usepackage{titlefoot}

\allowdisplaybreaks[1]

\newcommand{\dist}{\mathrm{dist}\,}          

\newcommand{\capE}{\mathcal{E}}
\newcommand{\capF}{\mathcal{F}}

\newcommand{\capH}{\mathcal{H}}
\newcommand{\capI}{\mathcal{I}}

\newcommand{\mathN}{\mathbb{N}}

\newcommand{\mathR}{\mathbb{R}}
\newcommand{\mathS}{\mathbb{S}}

\let\OLDthebibliography\thebibliography
\renewcommand\thebibliography[1]{
	\OLDthebibliography{#1}
	\setlength{\parskip}{1.5pt}
	\setlength{\itemsep}{1pt plus 0.3ex}
}

\makeatletter

\@addtoreset{equation}{section}
\makeatother

\theoremstyle{mystyle}
\newtheorem{theorem}{Theorem}[section]
\newtheorem*{theorem*}{Theorem}
\newtheorem{lemma}[theorem]{Lemma}
\newtheorem{proposition}[theorem]{Proposition}

\theoremstyle{definition}

\theoremstyle{remark}
\newtheorem{remark}[theorem]{Remark}

\begin{document}
\title{Existence of minimizers for a generalized liquid drop model with fractional perimeter}

\author{Matteo Novaga\footnote{Universit\`a di Pisa, Largo Bruno Pontecorvo 5, 56127 Pisa, Italy. E-mail: {\tt matteo.novaga@unipi.it}}
\and Fumihiko Onoue\footnote{Scuola Normale Superiore, Piazza Cavalieri 7, 56126 Pisa, Italy. E-mail: {\tt fumihiko.onoue@sns.it}}
}

\date{\today}	

\maketitle

\begin{abstract}
	We consider the minimization problem of the functional given by the sum of the fractional perimeter and a general Riesz potential, which is one generalization of Gamow's liquid drop model. We first show the existence of minimizers for any volumes if the kernel of the Riesz potential decays faster than that of the fractional perimeter. Secondly, we show the existence of generalized minimizers for any volumes if the kernel of the Riesz potential just vanishes at infinity. Finally, we study the asymptotic behavior of minimizers when the volume goes to infinity and we prove that a sequence of minimizers converges to the Euclidean ball up to translations if the kernel of the Riesz potential decays sufficiently fast.  
\end{abstract}

\tableofcontents

\section{Introduction}\label{sectionIntroduction}
We study existence and asymptotic behavior of minimizers for the minimization problem
\begin{equation}\label{minimizationGeneralizedFunctional}
	E_{s,g}[m] \coloneqq \inf\left\{\capE_{s,g}(E) \mid \text{$E \subset \mathR^N:$ measurable, $|E| = m$} \right\}
\end{equation}
for any $m>0$, where we define the functional $\capE_{s,g}$ as
\begin{equation}\label{generalizedFunctional}
	\capE_{s,g}(E) \coloneqq P_s(E) + \, V_g(E) 
\end{equation}
for any measurable set $E \subset \mathR^N$. Note that the first term $P_s$ of \eqref{generalizedFunctional} is the \textit{fractional $s$-perimeter} with $s \in (0,\,1)$ defined by
\begin{equation}\nonumber
	P_s(E) \coloneqq \int_{E}\int_{E^c}\frac{1}{|x-y|^{N+s}}\,dx\,dy
\end{equation}
for any measurable set $E \subset \mathR^N$, and the second term $V_g$ of \eqref{generalizedFunctional} is the generalized Riesz potential, defined by
\begin{equation}\nonumber
	V_g(E) \coloneqq \int_{E}\int_{E}g(x-y)\,dx\,dy
\end{equation}
for any measurable set $E \subset \mathR^N$, where $g: \mathR^N \setminus \{0\} \to \mathR$ is a non-negative, measurable, and radially symmetric function. The precise assumptions on $g$ will be given in Section \ref{sectionMainResults}.

Problem \eqref{minimizationGeneralizedFunctional} can be regarded as a nonlocal counterpart of the minimization problem
\begin{equation}\label{classicalLiquidDropModel}
	E_{g}[m] \coloneqq \inf\left\{ \capE_g(E) \coloneqq P(E) +  V_{g}(E) \mid \text{$E \subset \mathR^N$, $|E|=m$}\right\}
\end{equation}
where we let $P(E)$ be the classical perimeter of a set $E$.

A relevant physical case of Problem \eqref{classicalLiquidDropModel} is when $N=3$ and $g(x) = |x|^{-1}$ for $x\in\mathR^N\setminus\{0\}$. In this case, this problem is referred as Gamow's liquid drop model and was firstly investigated by George Gamow in \cite{Gamow} to reveal some basic properties of atoms and provide a simple model of the nuclear fission. In this model, an atomic nucleus can be regarded as nucleons (protons and neutrons) contained in a set $E \subset \mathR^N$. The nucleons are assumed to be concentrated with constant density and implies the number of nucleons is proportional to $|E|$. From a physical point of view, the classical perimeter term corresponds to surface tension, which is minimised by spherical nuclei. On the other hand, the Riesz potential corresponds to a Coulomb repulsion, which tends to drive nuclei away from each other. Due to these properties, the competition between the perimeter term and Riesz potential occurs. By rescaling, one can easily observe this phenomenon. Indeed, using the dilation $\lambda \mapsto \lambda \,E$ for a set $E$, we have that
\begin{equation}\nonumber
	P(\lambda E) + V_{g}(\lambda E) = \lambda^{N-1}\,P(E) +  \int_{E}\int_{E} \lambda^{2N}\,g(\lambda(x-y))\,dx\,dy 
\end{equation}
for any $\lambda>1$ and measurable set $E \subset \mathR^N$. Then, if the kernel $g$ satisfies $g(x) \thickapprox |x|^{-(N+\delta)}$ as $|x| \to \infty$ for $\delta < 1$, we have that $\lambda^{2N}\,g(\lambda\,x) \thickapprox \lambda^{N-\delta}$. Thus, the Riesz potential dominates the perimeter as $\lambda \to \infty$ since $\delta < 1$. On the other hand, if $\lambda \to 0$, then the perimeter dominates the Riesz potential. In Problem \eqref{minimizationGeneralizedFunctional}, the nonlocal perimeter $P_s$ with $s \in (0,\,1)$ behaves like the classical perimeter when $s$ approaches to 1 (see the asymptotic behavior of the $s$-fractional perimeter and more general results in \cite{ADPM, BBM, CaVa, Davila, LeSp01, LeSp02, Ponce}). The authors in \cite{CMT} published a survey on this model and the historical background and some references are therein.

Now let us briefly review the previous works on the classical liquid drop model. Recently, the authors in \cite{FrNa} revisited this model and some references are also therein. The main interest from the mathematical point of view is to investigate the following three topics: the existence of minimizer, the non-existence of minimizer, and the minimality of the ball. Kn\"upfer and Muratov in \cite{KnMu01, KnMu02} considered when $g$ is equal to $|x|^{-\alpha}$ for $\alpha \in (0,\,N)$ with $N \geq 2$ and proved that there exists constants $0< m_0 \leq m_1 \leq m_2 <\infty$ such that the following three things hold: if $N \geq 2$, $\alpha \in (0,\,N)$, and $m \leq m_1$ , then Problem \eqref{classicalLiquidDropModel} admits a minimizer; if $N \geq 2$, $\alpha \in (0,\,2)$, and $m > m_2$, then Problem \eqref{classicalLiquidDropModel} does not admit a minimizer; finally, if $m \leq m_0$, then the ball is the unique minimizer whenever either $N = 2$ and $\alpha \in (0,\,2)$, or $3 \leq N \leq 7$ and $\alpha \in (0,\,N-1)$. Later, Julin in \cite{Julin} proved that, if $N \geq 3$ and $g(x) = |x|^{-(N-2)}$, the ball is the unique minimizer of $\capE_g$ whenever $m$ is sufficiently small. Bonacini and Cristoferi in \cite{BoCr} studied the case of the full parameter range $N \geq 2$ and $\alpha \in (0,\,N-1)$ when $g(x) = |x|^{-\alpha}$. Moreover, for a small parameter $\alpha$, the authors in \cite{BoCr} gave a complete characterization of the ground state. Namely, they showed that, if $\alpha$ is sufficiently small, there exists a constant $m_c$ such that the ball is the unique minimizer of $\capE_g$ for $m \leq m_c$ and $\capE_g$ does not have minimizers for $m > m_c$. In a slightly different context, Lu and Otto in \cite{LuOt} showed the non-existence of minimizers for large volumes and that the ball is the unique minimizer for small volumes when $N = 3$ and $g(x) = |x|^{-1}$. The authors were motivated by the ionization conjecture and the energy that the authors studied includes background potential, which behaves like an attractive term. In the similar context to \cite{LuOt}, the authors in \cite{FNB} showed the non-existence of minimizers for large volumes. In contrast, the authors in \cite{ABCT} proved that a variant of Gamow's model including the background potential admits minimizers for any volume, due to the effects from the background potential against the Riesz potential. Very recently, Novaga and Pratelli in \cite{NoPr} showed the existence of generalized minimizers for the energy associated with $\capE_g$ for any volume. After this work, Carazzao, Fusco, and Pratelli in \cite{CFP} showed that the ball is the unique minimizer for small volumes in any dimensions and for a general function $g$. Concerning the behavior of (generalized) minimizers for large volumes, Pegon in \cite{Pegon} showed that, if the kernel $g$ decays sufficiently fast at infinity and if the volume is sufficiently large, then minimizers exist and converge to a ball, up to rescaling, when the volume goes to infinity. Shortly after, Merlet and Pegon in \cite{MePe} proved that, in dimension $N=2$, minimizers are actually balls for large enough volumes. 

One remarkable feature of our results is that some nonlocal effect from the fractional perimeter of $\capE_{s,g}$ enables us to obtain minimizers of Problem \eqref{minimizationGeneralizedFunctional} for any volumes. As we mentioned above, it is known that Problem \eqref{classicalLiquidDropModel} admits the ball as the unique minimizer for sufficiently small volumes if the kernel $g$ is a Riesz kernel. Moreover, the author in \cite{Rigot} proved the existence of minimizers of $\capE_g$ for any volumes if the kernel $g$ has a compact support. Even if the kernel $g$ does not have a compact support but, if $g$ decays sufficiently fast, the author in \cite{Pegon} recently showed the existence of minimizers of $\capE_g$ for sufficiently large volumes. On the other hand, in our problem, we reveal that, if the kernel $g$ does not have a compact support but decays sufficiently fast, then minimizers of $\capE_{s,g}$ exist for any volumes. Hence, unlike the cases studied in \cite{Rigot} and \cite{Pegon}, a sort of nonlocal contribution of the fractional perimeter can ensure the existence of minimizers for any volumes.

By a heuristic argument, one can observe that, if $g$ decays sufficiently fast, the fractional perimeter dominates the Riesz potential even if the volume is sufficiently large. Indeed, if $g(x) \lesssim |x|^{-(N+s')}$ and $s' > s$, then we obtain that
\begin{align}
	\capE_{s,g}(\lambda E) &= \lambda^{N-s}\,P_s(E) + \lambda^{2N}\,\int_{E}\int_{E}g(\lambda(x-y))\,dx\,dy \nonumber\\
	&= \lambda^{N-s} \left(P_s(E) + \lambda^{s-s'}\,\int_{E}\int_{E}\lambda^{N+s'}g(\lambda(x-y))\,dx\,dy \right)\nonumber
\end{align}
for any set $E \subset \mathR^N$ and $\lambda >0$. Since we assume that $s'>s$, the Riesz potential could be dominated by the nonlocal perimeter term as $\lambda \to \infty$. Thus, one natural question is what would be the behavior of the energy like in the case that the kernel $g$ behaves like the kernel $|x|^{-(N+s)}$ of the fractional perimeter $P_s$.

In this paper, we answer this question. More precisely, we obtain the existence of minimizers for any volume and we characterize the asymptotic behavior of minimizers as the volume goes to infinity, assuming that the kernel $g$ decays faster than the kernel of the fractional perimeter $P_s$. More precisely, we first prove the existence of minimizers of $\capE_{s,g}$ for any volume. To see this, we assume that $g$ is symmetric with respect to the origin, radially non-increasing, and decays faster than the kernel of the fractional perimeter $P_s$. For the details, we refer to Section \ref{sectionPreliminary}. The strategy of the proof is inspired by the concentration-compactness lemma by Lions \cite{Lions01, Lions02} and has been adapted by many authors (see, for instance, \cite{GoNo, dCNRV, CeNo} for topics closely related to ours). We will give some intuitive explanation of the strategy before proving the claim in Section \ref{sectionExisMiniFastDecayGAnyVol}.

Secondly, we prove the existence of generalized minimizers of a generalized functional $\widetilde{\capE}_{s,g}$, which we will define later, under the assumption that the kernel $g$ vanishes at infinity. It is easy to see that this assumption is weaker than the assumption that $g$ decays faster than the kernel of the nonlocal perimeter, which we imposed to prove the first result. For convenience, we here give the definitions of the generalized functional and generalized minimizers. For any $m>0$, we define a \textit{generalized functional} of $\capE_{s,g}$ over the family of sequences of the sets $\{E^k\}_{k\in\mathN}$ with $\sum_{k=1}^{\infty}|E^k| = m$ as 
\begin{equation}\label{defiGeneralziedFunctional}
	\widetilde{\capE}_{s,g}\left(\{E^k\}_{k\in\mathN}\right) \coloneqq \sum_{k=1}^{\infty}\capE_{s,g}(E^k).
\end{equation}
Then we consider the minimization problem
\begin{equation}\label{minimizationGeneralizedMinimizerFunctional}
	\inf\left\{\widetilde{\capE}_{s,g}\left(\{E^k\}_{k\in\mathN}\right) \mid \text{$E^k$: measurable for any $k$, $\sum_{k}|E^k| = m$} \right\}
\end{equation}
and show the existence of a minimizer of Problem \eqref{minimizationGeneralizedMinimizerFunctional} for any $m>0$. We call such a minimizer a \textit{generalized minimizer} of $\capE_{s,g}$. The precise statement will be given Theorem \ref{theoremExistGeneralizedMiniAnyVolume} in Section \ref{sectionMainResults}. The idea to prove our second result is to show the identity 
\begin{equation}\nonumber
	\inf\left\{\capE_{s,g}(E) \mid |E|=m \right\} = \inf\left\{\widetilde{\capE}_{s,g}(\{E^k\}_k) \mid \sum_{k=1}^{\infty} |E^k| = m \right\}
\end{equation}
for any $m>0$ and apply the same method which we use in the proof of our first result. 

Finally, we investigate the asymptotic behavior of minimizers as the volume goes to infinity, under the assumption that $g$ decays faster at infinity than the kernel $|x|^{-(N+s)}$ of the fractional perimeter $P_s$. Here we require an assumption on $g$ which is stronger than the one we assume in the existence result. To study the asymptotic behavior, we consider an equivalent minimization problem. More precisely, one can have two problems equivalent to $E_{s,g}[m]$ for $m>0$ under a proper decay assumption on $g$. Indeed, since the kernel $g$ is integrable over $\mathR^N$ under some proper assumptions, one can rewrite the Riesz potential as
\begin{equation}\nonumber
	\int_{E}\int_{E} g(x-y)\,dx\,dy = |E|\,\|g\|_{L^1(\mathR^N)} - \int_{E}\int_{E^c} g(x-y)\,dx\,dy
\end{equation}
for any measurable set $E \subset \mathR^N$ with $|E| < \infty$. Hence, the minimization problem \eqref{minimizationGeneralizedFunctional} becomes 
\begin{equation}\label{minimizationModifiedProblem}
	\widehat{E}_{s,g}[m] \coloneqq \inf\left\{ P_s(E) - \int_{E}\int_{E^c} g(x-y)\,dx\,dy \mid |E|=m \right\}
\end{equation}
for any $m>0$. Moreover, by rescaling, one can further modify the minimization problem \eqref{minimizationModifiedProblem} into the equivalent problem
\begin{equation}\label{minimizationScalingModifiedProbelm}
	\widehat{E}^{\lambda}_{s,g}(B_1) \coloneqq \inf\left\{ \widehat{\capE}^{\lambda}_{s,g}(F) \coloneqq P_s(F) - \int_{E}\int_{E^c} \lambda^{N+s}g(\lambda(x-y))\,dx\,dy \mid |F|=|B_1| \right\}
\end{equation}
for any $\lambda>0$. Note that we will revisit more precisely the notations \eqref{minimizationModifiedProblem} and \eqref{minimizationScalingModifiedProbelm} in Section \ref{sectionMainResults}. With this notation, our last theorem is as follows; suppose that $\{F_n\}_{n}$ is any sequence of the minimizers of $\widehat{\capE}^{\lambda_n}_{s,g}$ such that $\lambda_n \to \infty$  and $|F_n|=|B_1|$ for any $n$. Then we have that the full sequence satisfies
\begin{equation}\nonumber
	|F_n \Delta B_1| \xrightarrow[n \to \infty]{} 0
\end{equation}
up to translations.

The organization of this paper is as follows: in Section \ref{sectionMainResults}, we will state our main results, namely, the existence of minimizers, the existence of generalized minimizers, and the convergence of any sequence of rescaled minimizers to the ball. In Section \ref{sectionPreliminary}, we will give several preliminary properties of minimizers of our energy. In Section \ref{sectionExisMiniFastDecayGAnyVol}, we will prove the existence of minimizers for any volumes and, in Section \ref{sectionExistGeneMiniAnyVol}, we will prove the existence of generalized minimizers for any volumes. In Section \ref{sectionAsymptoticMiniLargeVol}, we will study the asymptotic behavior of rescaled minimizers as the volume goes to infinity. We will also give the $\Gamma$-convergence result for our energy.

\begin{center}
 \textbf{{\small Acknowledgments}}
\end{center}

The authors would like to thank Marc Pegon for fruitful discussions on this work and several comments on the first draft of our manuscript.

The authors were supported by the INDAM-GNAMPA and by the PRIN Project 2019/24 {\it Variational methods for stationary and evolution problems with singularities and interfaces}.

\section{Main reuslts}\label{sectionMainResults}
We start with the assumptions on the kernel $g$  of the Riesz potential in the energy $\capE_{s,g}$. Throughout this paper, we assume that $s \in (0,\,1)$ and $g:\mathR^N \setminus \{0\} \to \mathR$ is in $L^1_{loc}(\mathR^N)$ and not identically equal to zero. We consider the following conditions on $g$:
\begin{itemize}
	\item[(g1)] $g$ is non-negative and radially non-increasing, namely,
	\begin{equation}\nonumber
		g(\lambda\,x) \leq g(x) \quad \text{for $x \in \mathR^N \setminus \{0\}$ and $\lambda \geq 1$}.
	\end{equation}
	\item[(g2)] $g$ is symmetric with respect to the origin, namely, $g(-x) = g(x)$ for any $x\in\mathR^N \setminus \{0\}$.
\end{itemize}
When we prove the existence of minimizers of $\capE_{s,g}$ in Section \ref{sectionExisMiniFastDecayGAnyVol}, we further assume the following condition on $g$:
\begin{itemize}
	\item[(g3)] There exist constants $R_0>1$ and $\beta \in (0,\,1)$ such that
	\begin{equation}\nonumber
		g(x) \leq \frac{\beta}{|x|^{N+s}} \quad \text{for any $|x| \geq R_0$}.
	\end{equation}
	($g$ decays faster than the kernel of $P_s$ far away from the origin.)
\end{itemize}
On the other hand, when we prove the existence of generalized minimizers of $\widetilde{\capE}_{s,g}$ in Section \ref{sectionExistGeneMiniAnyVol}, we assume the following condition, weaker than $(\mathrm{g}3)$:
\begin{itemize}
	\item[(g4)] $g$ vanishes at infinity, namely, $g(x) \to 0$ as $|x| \to \infty$.  
\end{itemize}
Moreover, when we study the asymptotic behavior of rescaled minimizers with large volumes in Section \ref{sectionAsymptoticMiniLargeVol}, we further impose the following assumption on $g$:
\begin{itemize}
	\item[(g5)] There exists a constant $\gamma \in (0,\,1)$ such that
	\begin{equation}\nonumber
		g(x) \leq \frac{\gamma}{|x|^{N+s}} \quad 
		\text{for any $x \in \mathR^N \setminus \{0\}$}, \quad g(x) = o\left(\frac{1}{|x|^{N+s}}\right) \quad \text{as $|x| \to \infty$}.
	\end{equation}
\end{itemize}

\begin{remark}
	From the assumption that $g \in L^1_{loc}(\mathR^N)$, we can easily show that $V_g(B) < +\infty$ for any ball $B \subset\ \mathR^N$. Indeed, one may compute
	\begin{equation}\nonumber
		V_g(B) \leq \int_{B}\int_{2B(y)} g(x-y)\,dx\,dy = |B|\int_{2B(0)}g(x)\,dx < \infty.
	\end{equation}
	Moreover, if we assume $(\mathrm{g}3)$, we actually have that $g$ is integrable in $\mathR^N$. Indeed, since $g \in L^1_{loc}(\mathR^N)$, we have that $\|g\|_{L^1(B_{R_0})} < \infty$. On the other hand, from $(\mathrm{g}3)$ and the integrability of $|x|^{-(N+s)}$ in $B^c_{R_0}$, we also have that $\|g\|_{L^1(B^c_{R_0})} < \infty$. Hence, the claim holds true.
\end{remark}

\begin{remark}
	A condition ensuring assumption $(\mathrm{g}5)$ is the existence of constants $R_0>1$, $\gamma \in (0,\,1)$, and $t > s$ such that
	\begin{align}\label{assumptionGDecayInfinity}
		g(x) \leq  
		\begin{cases}
			\displaystyle 
			\frac{\gamma}{|x|^{N+s}} & \quad \text{if $0< |x| < R_0$}\\
			\displaystyle
			\frac{1}{|x|^{N+t}} & \quad \text{if $|x| \geq R_0$}.
		\end{cases}
	\end{align}
	Notice that this assumption is stronger than assumption $(\mathrm{g}3)$. To show that \eqref{assumptionGDecayInfinity} implies $(\mathrm{g}5)$, we first take any $\varepsilon>0$ and, without loss of generality, assume that $\varepsilon < R_0^{-(t-s)}$ where $R_0$ is as in \eqref{assumptionGDecayInfinity}. Then it holds that 
	\begin{equation}\label{assumptionGDecayInfinity02}
		\frac{1}{|x|^{N+t}} \leq \frac{\varepsilon}{|x|^{N+s}} \quad \text{for any $|x| \geq \varepsilon^{-\frac{1}{t-s}}$}
	\end{equation}
	and thus, from \eqref{assumptionGDecayInfinity} and \eqref{assumptionGDecayInfinity02}, we obtain that
	\begin{equation}\nonumber
		g(x) \leq \frac{\varepsilon}{|x|^{N+s}} 
	\end{equation}
	for any $|x| \geq \varepsilon^{-\frac{1}{t-s}}$. Note that we have used the assumption $t > s$. From \eqref{assumptionGDecayInfinity}, we can easily show that $g(x) \leq \gamma|x|^{-(N+s)}$ for any $x \neq 0$. Hence, this completes the proof of the claim.
\end{remark}

\begin{remark}
	In the case of Problem \eqref{minimizationScalingModifiedProbelm} for large volumes, the author in \cite{Pegon} assumed that the kernel $g$ satisfies
	\begin{equation}\label{assumptionKernelG}
		g \in L^1(\mathR^N), \quad \int_{\mathR^N}|x| \,g(x)\,dx < + \infty.
	\end{equation}
	This condition with the radial symmetry of $g$ implies that $g$ satisfies 
	\begin{equation}
		g(x) \leq \frac{1}{|x|^{N+s}} \quad \text{for $|x|<1$}, \quad g(x) \leq \frac{c(g)}{|x|^{N+1}} \quad 
		\text{for $|x|>1$}
	\end{equation}
	where $c(g)>0$ is some constant. One may find the proof of this implication, for instance, in \cite{Carazzato}. Hence, it is easy to see that the assumption \eqref{assumptionKernelG} implies our assumption $(\mathrm{g}5)$. 
\end{remark}

Now we can state the main results of this paper. In the first result, we show the existence of minimizers of $\capE_{s,g}$ for any volume under the assumption that $g$ decays faster than the kernel of $P_s$ at infinity.
\begin{theorem}\label{theoremExistMiniAnyVolumeFasterDecay}
	Assume that the kernel $g: \mathR^N \setminus \{0\} \to \mathR$ satisfies the assumptions $(\mathrm{g}1)$, $(\mathrm{g}2)$, and $(\mathrm{g}3)$. Then, there exists a minimizer of $\capE_{s,g}$ with the volume $m$ for any $m>0$. 
	
	Moreover, the boundary of every minimizer has the regularity of class $C^{1,\alpha}$ with $\alpha \in (0,\,1)$ except a closed set of Hausdorff dimension $N-3$. 
\end{theorem}
The proof is inspired by so-called the ``concentration-compactness" lemma by Lions in \cite{Lions01, Lions02} and we apply the same idea shown in \cite{dCNRV}. We will roughly explain the idea of the proof in Section \ref{sectionExisMiniFastDecayGAnyVol}.

In the second theorem, we show the existence of generalized minimizer of $\widetilde{\capE}_{s,g}$ for any volume, under the assumption that $g$ vanishes at infinity. Notice that this assumption is weaker than the one we impose in Theorem \ref{theoremExistMiniAnyVolumeFasterDecay}.
\begin{theorem}\label{theoremExistGeneralizedMiniAnyVolume}
	Assume that the kernel $g: \mathR^N \setminus \{0\} \to \mathR$ satisfies the assumptions $(\mathrm{g}1)$, $(\mathrm{g}2)$, and $(\mathrm{g}4)$. Then, there exists a generalized minimizer of $\widetilde{\capE}_{s,g}$ for any $m>0$, namely, there exist a number $M \in \mathN$ and a sequence of sets $\{E^k\}_{k\in\mathN}$ such that
	\begin{equation}\nonumber
		\sum_{k=1}^{M} \capE_{s,g}(E^k) = \inf\left\{\widetilde{\capE}_{s,g}(\{E^k\}_k) \mid \sum_{k=1}^{M}|E^k| = m \right\},
	\end{equation}
	and $E^k$ is also a minimizer of $\capE_{s,g}$ among sets of volume $|E^k|$ for every $k\in\mathN$.
\end{theorem}
As we mentioned in Section \ref{sectionIntroduction}, the idea of the proof is based on the observation that Problem \eqref{minimizationGeneralizedMinimizerFunctional} can be reduced into Problem \eqref{minimizationGeneralizedFunctional}. 

Finally, we study the asymptotic behavior of minimizers of $\capE_{s,g}$ when the volume goes to infinity, under the assumption that $g$ decays much faster than the kernel $|x|^{-(N+s)}$ of $P_s$ far away from the origin. 

Before stating the theorem, in order to study the behavior of the minimizers of the minimization problem $E_{s,g}[m]$ for any $m>0$, it is convenient to lift the volume constraint onto the functional itself and work with fixed volume $|B_1|$. To see this, we first define a rescaled kernel by 
\begin{equation}\label{rescaledKernel}
	g_{\lambda}(x) \coloneqq \lambda^{N+s} \, g(\lambda\,x)
\end{equation}
for any $x \neq 0$ and $\lambda>0$. Then we show the equivalence of the rescaled problem in the following proposition.
\begin{proposition}[Equivalent problem]\label{propositionRescaledProblem}
	Let $m>0$. Assume that the kernel $g : \mathR^N \setminus \{0\} \to \mathR$ is in $L^1_{loc}(\mathR^N)$. Then, setting $\lambda^N \coloneqq m\,|B_1|^{-1}$, we have that the problem $E_{s,g}[m]$ is equivalent to 
	\begin{equation}\nonumber
		E^{\lambda}_{s,g}(B_1) \coloneqq \inf\left\{ P_s(F) + V_{g_\lambda}(F) \mid \text{$F \subset \mathR^N:$ measurable, $|F| = |B_1|$} \right\}
	\end{equation}
	where $g_{\lambda}$ is given in \eqref{rescaledKernel}.
	
	Moreover, under the assumption that $g$ is integrable on $\mathR^N$, the minimization problem $E_{s,g}[m]$ is also equivalent to Problem \eqref{minimizationScalingModifiedProbelm}.
\end{proposition}
\begin{proof}
	Given any $E$ with $|E|=m$ and setting $F \coloneqq \lambda^{-1}\,E$ where $\lambda^N = m\,|B_1|^{-1}$, we have that $|F| = |B_1|$ and 
	\begin{align}\label{rescaledIdenityEnergy}
		\capE_{s,g}(E) &= \lambda^{N-s}\,P_s(F) + \lambda^{2N}\,\int_{F}\int_{F}g(\lambda(x-y))\,dx\,dy \nonumber\\
		&= \lambda^{N-s} \left( P_s(F) + \int_{F}\int_{F}\lambda^{N+s}g(\lambda(x-y))\,dx\,dy \right) \nonumber\\
		&= \lambda^{N-s} \left( P_s(F) + V_{g_\lambda}(F) \right)
	\end{align}
	where $g_{\lambda}(x) \coloneqq \lambda^{N+s}\,g(\lambda\,x)$ as in \eqref{rescaledKernel}. For the latter part of the claim, we first recall the equivalent minimization problem
	\begin{equation}\nonumber
		\widehat{E}_{s,g}[m] \coloneqq \inf \left\{P_s(E) - \int_{E}\int_{E^c}g(x-y)\,dx\,dy \right\},
	\end{equation}
	which is equivalent to the problem $E_{s,g}[m]$ for any $m>0$. Thus, from \eqref{rescaledIdenityEnergy}, we obtain that
	\begin{equation}\nonumber
		\capE_{s,g}(E) = \lambda^{N-s} \left( P_s(F) - \int_{F}\int_{F^c}g_{\lambda}(x-y)\,dx\,dy + m\,\|g\|_{L^1(\mathR^N)} \right).
	\end{equation}
	Hence, we conclude that the claim is valid.  
\end{proof}

Now we are prepared to state the last theorem of this present paper.
\begin{theorem}\label{theoremAsympMiniLargeVolume}
	Let $s\in(0,\,1)$ and $\{\lambda_n\}_{n\in\mathN} \subset (1,\,\infty)$ with $\lambda_n \to \infty$ as $n \to \infty$. Let $\{F_n\}_{n\in\mathN}$ be a sequence of minimizers for $\widehat{\capE}^{\lambda_n}_{s,g}$ with $|F_n|=|B_1|$ for each $n\in\mathN$. Assume that the kernel $g: \mathR^N \setminus \{0\} \to \mathR$ is radially symmetric and satisfies the assumptions $(\mathrm{g}1)$, $(\mathrm{g}2)$, and $(\mathrm{g}5)$. Then, the sequence $\{F_n\}_{n\in\mathN}$ converges to the unit ball $B_1$, up to translations, in the sense of $L^1$-topology, namely, 
	\begin{equation}\nonumber
		|F_n \Delta B_1| \xrightarrow[n \to \infty]{} 0.
	\end{equation}
\end{theorem}
\begin{remark}
	In this paper, we basically assume that the kernel $g$ is locally integrable in $\mathR^N$, especially near the origin; however, Theorem \ref{theoremAsympMiniLargeVolume} is still valid even if $g$ is not integrable in the ball centred at the origin. This is because the assumption that $g(x) \leq |x|^{-(N+s)}$ for $x \neq 0$ is sufficient enough for the nonlocal perimeter $P_g$ to be finite for any ball $B$. We emphasize that the local integrability of $g$ ensures that Problem \eqref{minimizationScalingModifiedProbelm} is equivalent to Problem \eqref{minimizationGeneralizedFunctional}, which is a nonlocal and generalized version of the liquid drop model by Gamow. 
\end{remark}
The idea of the proof is based on the same argument of Theorem \ref{theoremExistMiniAnyVolumeFasterDecay} and the $\Gamma$-convergence result on the energy $\widehat{\capE}^{\lambda}_{s,g}$ as $\lambda \to \infty$. We will give the precise strategy of the proof in Section \ref{sectionAsymptoticMiniLargeVol} and, for the $\Gamma$-convergence result, the readers should refer to Proposition \ref{propositionGammaConvergenceNonlocalEnergy} in Section \ref{sectionAsymptoticMiniLargeVol}.

\section{Preliminary results for minimizers of $\capE_{s,g}$}\label{sectionPreliminary}
In this section, we collect several properties for minimizers of $\capE_{s,g}$ under the assumptions on $g$ in Section \ref{sectionMainResults} 

First of all, we recall one important property on the fractional perimeter $P_s$ with $0<s<1$.
\begin{proposition}\label{propositionIntersectionConvexSmaller}
	For any $s \in (0,\,1)$ and measurable set $E \subset \mathR^N$ with $|E|<\infty$, it follows that $P_s(E \cap K) \leq P_s(E)$ for every convex set $K \subset \mathR^N$. 
\end{proposition}
The proof can be found in \cite[Lemma B.1]{FFMMM} and we do not give a proof of this proposition here. We also refer to \cite[Corollary 5.3]{CRS} and \cite{ADPM} for related properties to Proposition \ref{propositionIntersectionConvexSmaller}. 

The assumption that $g$ is radially non-increasing enables us to show the scaling property of $\capE_{s,g}$ by simple computations.
\begin{lemma}[Scaling lemma]\label{lemmaScalingEnergy}
	Let $E \subset \mathR^N$ be a measurable set with $|E| <\infty$. Assume that the kernel $g:\mathR^N \setminus \to \mathR$ satisfies $(\mathrm{g}1)$ and $(\mathrm{g}2)$. Then, for any $\lambda \geq 1$, it follows that
	\begin{equation}\nonumber
		\capE_{s,g}(\lambda \,E) \leq \lambda^{2N} \capE_{s,g}(E).
	\end{equation}
\end{lemma}
\begin{proof}
	From the change of variables and the choice of $\lambda>1$, we have
	\begin{equation}\label{scalingEsti01}
		P_s(\lambda\,E) = \int_{\lambda\,E}\int_{\lambda\,E^c}\frac{dx\,dy}{|x-y|^{N+s}} = \lambda^{N-s} \int_{E}\int_{E^c} \frac{dx\,dy}{|x-y|^{N+s}} = \lambda^{N-s} P_s(E) \leq \lambda^{2N} P_s(E)
	\end{equation}
	for any $E \subset \mathR^N$ and $\lambda >1$. From the assumptions on $g$ and the change of variables again, we can compute the Riesz potential as follows:
	\begin{equation}\label{scalingEsti02}
		V_g(\lambda\,E) = \lambda^{2N} \int_{E}\int_{E} g(\lambda(x-y))\,dx\,dy \leq  \lambda^{2N} \int_{E}\int_{E} g(x-y)\,dx\,dy = \lambda^{2N} V_g(E)
	\end{equation}
	for any $E \subset \mathR^N$ with $|E| <\infty$ and $\lambda>1$. Therefore, from \eqref{scalingEsti01} and \eqref{scalingEsti02}, we obtain
	\begin{equation}\nonumber
		\capE_{s,g}(\lambda \,E) \leq \lambda^{2N} (P_s(E) +  V_g(E)) \leq \lambda^{2N} \capE_{s,g}(E)
	\end{equation}
	and this completes the proof.
\end{proof}

We next prove the boundedness of minimizers of $\capE_{s,g}$ among sets of volume $m$.
\begin{lemma}[Boundedness of minimizers] \label{lemmaBoundednessMinimizers}
	Let $m>0$. Assume that the kernel $g$ satisfies the conditions $(\mathrm{g}1)$ and $(\mathrm{g}2)$. If $E \subset \mathR^N$ is a minimizer of $\capE_{s,g}$ with the volume $m$, then $E$ is bounded up to negligible sets, namely, there exists a constant $\hat{R}>0$ such that $|E \setminus B_{\hat{R}}(0)| = 0$.
\end{lemma}
\begin{proof}
	Let $E$ be a minimizer of $\capE_{s,g}$ with $|E|=m$. By setting $\phi(r) \coloneqq |E \setminus B_r(0)|$ for any $r>0$, we have that $\phi^{\prime}(r) = -\capH^{N-1}( E \cap \partial B_r(0) )$ for a.e. $r>0$. In order to prove the claim, we suppose by contradiction that $\phi(r) > 0$ for any $r>0$. Setting $E_r \coloneqq E \cap B_r(0)$ for any $r>0$ and $\lambda_r \coloneqq \frac{m}{m - \phi(r)}$ for any $r>0$, then we choose $\lambda_r  E_r$ as the competitor of $E$ if $\phi(r) < m$ and thus, we have that
	\begin{align}\label{estimateMinimalityForBoundedness}
		\capE_{s,g}(E) \leq \capE_{s,g}(\lambda_r  E_r) &\leq (\lambda_r)^{N-s} P_s(E_r) + (\lambda_r)^{2N} V_{g_{\lambda_r}}(E_r) \nonumber\\
		&\leq \capE_{s,g}(E_r) + \left((\lambda_r)^{N-s} - 1\right)P_s(E_r) + \left((\lambda_r)^{2N} - 1\right) V_g(E_r).
	\end{align}
	Since $\phi(r) \to 0$ as $r \to \infty$, we can choose a constant $R_0>0$ such that $\phi(r) \leq m/2$ for any $r \geq R_0$ and thus, we may assume that
	\begin{equation}\label{estimateMinimalityForBoundedness02}
		(\lambda_r)^{N-s} -1 \leq c_0\,\phi(r), \quad (\lambda_r)^{2N} - 1 \leq c'_0\,\phi(r)
	\end{equation}
	for any $r \geq R_0$ where $c_0$ and $c'_0$ are some positive constants depending only on $N$, $s$, and $m$. Then, by using the decomposition property of $P_s$ and $V_g$ and combining \eqref{estimateMinimalityForBoundedness02} with \eqref{estimateMinimalityForBoundedness}, we have that 
	\begin{align}\label{estimateMinimalityForBoundedness03}
		P_s(E \setminus B_r(0)) &\leq P_s(E \setminus B_r(0)) + V_g(E \setminus B_r(0)) \nonumber\\
		&\leq 2\int_{E \cap B_r(0)}\int_{E \setminus B_r(0)}\frac{dx\,dy}{|x-y|^{N+s}} + c_0\,\phi(r)\,P_s(E_r) + c'_0\,\phi(r)\,V_g(E_r) 
	\end{align}
	for any $r \geq R_0$. From Proposition \ref{propositionIntersectionConvexSmaller} and the definition of $V_g$, we have that, for any $r>0$,
	\begin{equation}\nonumber
		P_s(E_r) + V_g(E_r) \leq P_s(E) + V_g(E) = E_{s,g}[m].
	\end{equation}
	Thus, from \eqref{estimateMinimalityForBoundedness03}, we obtain that
	\begin{equation}\nonumber
		P_s(E \setminus B_r(0)) \leq 2\int_{E \cap B_r(0)}\int_{E \setminus B_r(0)}\frac{dx\,dy}{|x-y|^{N+s}} + (c_0+c'_0)E_{s,g}[m]\,\phi(r)
	\end{equation}
	for any $r \geq R_0$. Now using the isoperimetric inequality of $P_s$ and the fact that $E \cap B_r(0) \subset B^c_{|y|-r}(y)$ for any $y \in E \setminus B_r(0)$, we obtain
	\begin{align}\label{estimateMinimality}
		\frac{P_s(B_1)}{|B_1|^{\frac{N-s}{N}}}\,\phi(r)^{\frac{N-s}{N}} &\leq 2\int_{E \setminus B_r(0)}\int_{B^c_{r-|y|}(y)}\frac{dx\,dy}{|x-y|^{N+s}} + (c_0+c'_0)E_{s,g}[m]\,\phi(r) \nonumber\\
		&= \frac{2|\partial B_1|}{s}\int_{E \setminus B_r(0)}\int_{|y|-r}^{\infty}\frac{1}{t^{1+s}}\,dt\,dy + (c_0+c'_0)E_{s,g}[m]\,\phi(r) \nonumber\\
		&= \frac{2|\partial B_1|}{s}\int_{E \setminus B_r(0)}\frac{1}{(|y|-r)^s}\,dy + (c_0+c'_0)E_{s,g}[m]\,\phi(r) \nonumber\\
		&= \frac{2|\partial B_1|}{s} \int_{r}^{\infty}\frac{-\phi^{\prime}(\sigma)}{(\sigma - r)^s}\,d\sigma + (c_0+c'_0)E_{s,g}[m]\,\phi(r)
	\end{align}
	for any $r>0$. Here we have used the co-area formula in the last equality. Since $\phi$ is non-increasing, there exists a constant $R'_0=R'_0(N,s,m)>0$ such that
	\begin{equation}\label{absorptionTerm}
		(c_0+c'_0)E_{s,g}[m]\,\phi(r) \leq 	\frac{P_s(B_1)}{2|B_1|^{\frac{N-s}{N}}}\,\phi(r)^{\frac{N-s}{N}}
	\end{equation}
	for any $r \geq \max\{R_0, \, R'_0\}$. From \eqref{estimateMinimality} and \eqref{absorptionTerm}, we obtain
	\begin{equation}\label{estimateMinimalityKey}
		c_1\,\phi(r)^{\frac{N-s}{N}} \leq c_2\,\int_{r}^{\infty}\frac{-\phi^{\prime}(\sigma)}{(\sigma - r)^s}\,d\sigma
	\end{equation}
	for any $r \geq \max\{R_0, \, R'_0\}$ where we set $c_1 \coloneqq (2|B_1|^{\frac{N-s}{N}})^{-1}\,P_s(B_1)$ and $c_2 \coloneqq 2s^{-1}|\partial B_1|$. By integrating the both sides in \eqref{estimateMinimalityKey} over $r \in [R,\,\infty)$ for any fixed constant $R \geq \max\{R_0, \, R'_0\}$ and changing the order of the integration, we obtain
	\begin{align}\label{keyEstiBoundednessMini}
		c_1\,\int_{R}^{\infty}\phi(r)^{\frac{N-s}{N}}\,dr \leq c_2\,\int_{R}^{\infty}\int_{r}^{\infty}\frac{-\phi^{\prime}(\sigma)}{(\sigma - r)^s}\,d\sigma\,dr &= c_2\,\int_{R}^{\infty}\int_{R}^{\sigma} \frac{-\phi^{\prime}(\sigma)}{(\sigma - r)^s}\,dr\,d\sigma \nonumber\\
		&= -\frac{c_2}{1-s}\, \int_{R}^{\infty}\phi^{\prime}(\sigma)\,(\sigma - R)^{1-s}\,d\sigma.
	\end{align}
	Hence, by employing the same argument shown in \cite[Lemma 4.1]{dCNRV} and \cite[Proposition 3.2]{CeNo} together with \eqref{keyEstiBoundednessMini}, we obtain that $\phi(R) = 0$, which contradicts the assumption that $\phi(r)>0$ for any $r>0$. Therefore, we conclude the existence of the constant $\hat{R}>0$ such that $|E \setminus B_{\hat{R}}| = 0$.
\end{proof}

Next, by using assumption $(\mathrm{g}4)$, we show the sub-additivity result of the function $m \mapsto E_{s,g}[m]$. We recall that $E_{s,g}[m]$ is defined by
\begin{equation}\nonumber
	\inf\left\{ \capE_{s,g}(E) \mid \text{$E \subset \mathR^N$: measurable, $|E|=m$} \right\}.
\end{equation}
for any $m>0$.
\begin{lemma}[Sub-additivity of $E_{s,g}$]\label{lemmaSubadditivityEnergy}
	Let $m>0$ be any number. Assume that the kernel $g:\mathR^N \setminus \{0\} \to \mathR$ satisfies $(\mathrm{g}1)$, $(\mathrm{g}2)$, and $(\mathrm{g}4)$. Then, for any $m_1 \in (0,\,m]$, it holds
	\begin{equation}\nonumber
		E_{s,g}[m] \leq E_{s,g}[m_1] + E_{s,g}[m - m_1].
	\end{equation}  
\end{lemma}
\begin{proof}
	The idea is in the same spirit as the one shown in \cite[Lemma 3]{LuOt} (see also \cite{Onoue}).
	
	Let $m > 0$ be any constant and we take any $m_1 \in (0,\,m)$. By definition, for any $\eta>0$, there exist measurable sets $E_1,\,E_2 \subset \mathbb{R}^N$ with the volume constraints $|E_1|=m_1$ and $|E_2|=m - m_1$ such that
	\begin{equation}\label{minimizingSeqInequality}
		\capE_{s,g}(E_1) + \capE_{s,g}(E_2) \leq E_{s,g}[m_1]+ E_{s,g}[m_2] + \eta.
	\end{equation}
	Now we may assume that $E_1$ and $E_2$ are bounded. Indeed, we can observe that the minimum of $\capE_{s,g}$ among unbounded sets of volume $m$ is not smaller than the minimum of $\capE_{s,g}$ among bounded sets of volume $m$. To see this, for any unbounded set $E$ with $|E|=m$, we can choose sufficiently large $R>1$ in such a way that $|E \setminus B_R(0)|$ is as small as possible. Then, setting $\widehat{E} \coloneqq \lambda(R)\,(E \cap B_R(0))$ where $\lambda(R)^N \coloneqq \frac{m}{m - |E \setminus B_R(0)|} \geq 1$, we obtain, from Lemma \ref{lemmaScalingEnergy}, that
	\begin{equation}\nonumber
		|\widehat{E}| = \lambda(R)^N\,(m - |E \setminus B_R(0)|) = m
	\end{equation}
	and 
	\begin{align}\label{keyEstiEnergyModifiedSet}
		\capE_{s,g}(\widehat{E}) &\leq \lambda(R)^{2N} \capE_{s,g}(E \cap B_R(0)) \nonumber\\
		&\leq \lambda(R)^{2N} \capE_{s,g}(E) - P_s(E \setminus B_R(0)) + 2\int_{E \cap B_R(0)}\int_{E \setminus B_R(0)}\frac{dx\,dy}{|x-y|^{N+s}}.
	\end{align}
	Here we have used the following identity of the nonlocal perimeter:
	\begin{equation}\nonumber
		P_s(E \cup F) = P_s(E) + P_s(F) - 2\int_{E}\int_{F}\frac{1}{|x-y|^{N+s}}\,dx\,dy \nonumber
	\end{equation}
	for any measurable sets $E,\,F \subset \mathR^N$. From the isoperimetric inequality and the computation in \eqref{estimateMinimality} in Lemma \ref{lemmaBoundednessMinimizers}, we have that
	\begin{equation}\label{keyEstiEnergyModifiedSet02}
		\capE_{s,g}(\widehat{E}) \leq \lambda(R)^{2N} \capE_{s,g}(E) - C_1\,|E \setminus B_R(0)|^{\frac{N-s}{N}} + C_2\, \int_{R}^{\infty}\frac{\capH^{N-1}(E \cap \partial B_{\sigma}(0))}{(\sigma - R)^s}\,d\sigma
	\end{equation}
	where we set $C_1 \coloneqq P_s(B_1)\,|B_1|^{-\frac{N-s}{N}}$ and $C_2 \coloneqq 2s^{-1}|\partial B_1|$. Since $E$ is unbounded, we have that the function $R \mapsto |E \setminus B_R(0)|$ is non-increasing and not equal to zero for any $R>0$. Thus, by applying the same argument in Lemma \ref{lemmaBoundednessMinimizers}, we can find that there exists a sequence $\{R_i\}_{i\in\mathN}$ such that $R_i \to \infty$ as $i \to \infty$ and 
	\begin{equation}\label{negativityRemainderTerm}
		- C_1\,|E \setminus B_{R_i}(0)|^{\frac{N-s}{N}} + C_2\, \int_{R_i}^{\infty}\frac{\capH^{N-1}(E \cap \partial B_{\sigma}(0))}{(\sigma - R_i)^s}\,d\sigma < 0
	\end{equation}
	for any $i\in\mathN$. Hence, from \eqref{keyEstiEnergyModifiedSet02} and \eqref{negativityRemainderTerm}, it follows that
	\begin{equation}\nonumber
		\inf\{\capE_{s,g}(E) \mid \text{$E$: measurable \& bounded, $|E|=m$}\} \leq  \capE_{s,g}(\widehat{E}) 
		< \lambda(R_i)^{2N} \capE_{s,g}(E)
	\end{equation} 
	for any $i\in\mathN$. From the fact that $\lambda(R_i) \to 1$ as $i \to \infty$, the arbitrariness of $E$ and by letting $i \to \infty$, we finally obtain that
	\begin{equation}\nonumber
		\inf\{\capE_{s,g}(E) \mid \text{$E$: bounded, $|E|=m$}\} \leq \inf\{\capE_{s,g}(E) \mid \text{$E$: unbounded, $|E|=m$}\},
	\end{equation}
	as we desired. 
	
	Now we focus on the case that both $E_1$ and $E_2$ are bounded. Since $E_1,\,E_2$ are bounded, we can find a vector $e \in \mathS^{N-1}$ such that it follows that
	\begin{equation}\nonumber
		\dist(E_1,\,(E_2 + d\,e)) \xrightarrow[d \to \infty]{} \infty.
	\end{equation}
	Then we may compute the energy as follows:
	\begin{align}
		\capE_{s,g}(E_1\cup(E_2+d\,e)) &= P_s(E_1\cup(E_2+d\,e)) + V_g(E_1\cup(E_2+d\,e)) \nonumber\\
		&\leq P_s(E_1) + P_s(E_2+d\,e) \nonumber\\
		&\qquad + V_g(E_1) + V_g(E_2+d\,e) + 2\int_{E_1}\int_{E_2+d\,e} g(x-y)\,dx\,dy \nonumber\\
		&\leq \capE_{s,g}(E_1) +\capE_{s,g}(E_2) + 2\int_{E_1}\int_{E_2+d\,e} g(x-y)\,dx\,dy. \nonumber
	\end{align}
	Here we have used the translation invariance of $P_s$ and $V_g$. From assumption $(\mathrm{g}4)$, which says that $g$ vanishes at infinity, we can show that
	\begin{equation}\nonumber
		\int_{E_1}\int_{E_2+d\,e} g(x-y)\,dx\,dy \xrightarrow[d \to \infty]{} 0.
	\end{equation}
	Since $|E_1 \sqcup (E_2 + d\,e)| = |E_1| + |E_2| = m$ for sufficiently large $d>0$ and from \eqref{minimizingSeqInequality}, we obtain
	\begin{equation}\nonumber
		E_{s,g}[m_1 + m - m_1] \leq E_{s,g}[m_1]+ E_{s,g}[m - m_1] + \eta + o(1).
	\end{equation}
	Letting $d \to \infty$ and then $\eta \to 0$, we conclude that the lemma is valid. 
\end{proof}

Next we prove the uniform density estimate of minimizers of $\capE_{s,g}$ for large $m>0$. 
\begin{lemma}[Uniform density estimate]\label{lemmaUniformDensity}
	Let $m \geq 1$. We assume that the kernel $g :\mathR^N \setminus \{0\} \to \mathR$ is integrable in $\mathR^N$ and satisfies the assumptions $(\mathrm{g}1)$ and $(\mathrm{g}2)$. Then there exist constants $c_0>0$ and $r_0>0$, depending only on $N$, $s$, and $g$, such that, if $E$ is a minimizer of $\capE_{s,g}$ with $|E| = m$, then it holds that
	\begin{equation}\label{uniformDensityEstimate}
		|E \cap B_r(x)| \geq c_0\, r^{N}
	\end{equation}
	for any $r\in(0,\,r_0]$ and $x\in\mathR^N$ with $|E \cap B_r(x)|>0$ for any $r>0$.
\end{lemma}
\begin{remark}
	Notice that, directly from Lemma \ref{lemmaUniformDensity}, we fail to obtain the uniform density estimates of minimizers of $\widehat{\capE}_{s,g}$. Indeed, when we use the dilation $\lambda \mapsto \lambda \,E$ of a minimizer $E$ in \eqref{uniformDensityEstimate} in such a way that $|\lambda \, E| = |B_1|$, the constant $r_0$ in Lemma \ref{lemmaUniformDensity} vanishes as $\lambda  \to \infty$.  
\end{remark}
\begin{proof}
	Let $E$ be a minimizer of $\capE_{s,g}$ with $|E|=m$ and $x_0 \in E$ be any point such that $|E \cap B_r(x_0)| > 0$ for any $r>0$. We set $\lambda_r^N \coloneqq \frac{m}{m - |E \cap B_r(x_0)|} \geq 1$ for any $r>0$. We may assume that $|E \cap B_r(x_0)| < m$ for any $0<r<1$. Then, from the minimality of $E$, we have the inequality 
	\begin{equation}\nonumber
		\capE_{s,g}(E) \leq \capE_{s,g}\left(\lambda_r (E \setminus B_r(x_0)) \right)
	\end{equation}
	for any $r >0$. We now recall the following identity on the nonlocal perimeter:
	\begin{equation}\label{decompositionNonlocalPeri}
		P_s(E \cup F) = P_s(E) + P_s(F) - 2\int_{E}\int_{F}\frac{1}{|x-y|^{N+s}}\,dx\,dy
	\end{equation}
	for any measurable sets $E,\,F\subset \mathR^N$ with $E \cap F = \emptyset$. Then, from Lemma \ref{lemmaScalingEnergy}, \eqref{decompositionNonlocalPeri}, and the fact that $V_g(E \setminus B_r(x_0)) \leq V_g(E)$, we have that
	\begin{align}\label{keyDensityEstimate}
		\capE_{s,g}(E) &\leq \capE_{s,g}(E \setminus B_r(x_0)) + (\lambda_r^{2N}-1)\,\capE_{s,g}(E \setminus B_r(x_0)) \nonumber\\
		&\leq \capE_{s,g}(E \setminus B_r(x_0)) + (\lambda_r^{2N}-1)\,\Big(\capE_{s,g}(E) - P_s(E \cap B_r(x_0))  \nonumber\\
		&\qquad \left. + \int_{E \cap B_r(x_0)}\int_{E \setminus B_r(x_0)} \frac{2\,dx\,dy}{|x-y|^{N+s}} \right) 
	\end{align}
	for $0<r<1$. Similarly to \eqref{decompositionNonlocalPeri}, we also have the following identity on the Riesz potential:
	\begin{equation}\label{decompositionRiesz}
		V_g(E \cup F) = V_g(E) + V_g(F) + 2\int_{E}\int_{F}g(x-y)\,dx\,dy
	\end{equation}
	for any measurable sets $E,\,F\subset \mathR^N$ with $E \cap F = \emptyset$. Thus, from \eqref{decompositionNonlocalPeri}, \eqref{keyDensityEstimate}, and \eqref{decompositionRiesz}, we further have that 
	\begin{align}\label{keyDensityEstimate02}
		\lambda_r^{2N} \, P_s(E \cap B_r(x_0)) 
		&\leq \lambda_r^{2N} \, \int_{E \cap B_r(x_0)}\int_{E \setminus B_r(x_0)} \frac{2\,dx\,dy}{|x-y|^{N+s}} + (\lambda_r^{2N} - 1)\,E_{s,g}[m]
	\end{align}
	for any $0 < r < 1$. Recalling the definition of $\lambda_r$, we have that
	\begin{equation}\label{defLambda}
		\lambda_r^{2N} = \frac{m^2}{(m - |E \cap B_r|)^2}, \quad \lambda_r^{2N} - 1 = \frac{|E \cap B_r|}{m - |E \cap B_r|} \left( 2 +  \frac{|E \cap B_r|}{m - |E \cap B_r|} \right)
	\end{equation}
	for any $0 < r < 1$. From \eqref{keyDensityEstimate02} and \eqref{defLambda}, we finally obtain
	\begin{equation}\nonumber
		P_s(E \cap B_r(x_0)) 
		\leq \int_{E \cap B_r(x_0)}\int_{E \setminus B_r(x_0)} \frac{2\,dx\,dy}{|x-y|^{N+s}} + \frac{2E_{s,g}[m]}{m} |E \cap B_r(x_0)|
	\end{equation}
	for any $0 < r <1$. Hence, from the nonlocal isoperimetric inequality, we have that
	\begin{align}\label{densityEstimate01}
		&\frac{P_s(B_1)}{|B_1|^{\frac{N-s}{N}}} |E
		\cap B_r(x_0)|^{\frac{N-s}{N}} \nonumber\\
		&\leq P_s(E \cap B_r(x_0)) + V_g(E \cap B_r(x_0)) + 2\int_{E \cap B_r(x_0)}\int_{E \setminus B_r(x_0)}g(x-y)\,dx\,dy \nonumber\\
		&\leq 2 \int_{E \cap B_r(x_0)}\int_{E \setminus B_r(x_0)} \frac{1}{|x-y|^{N+s}}\,dx\,dy + \frac{2 \, E_{s,g}[m]}{m}|E \cap B_r(x_0)| 
	\end{align}
	for any small $r>0$. Noticing that $E \setminus B_r(x_0) \subset 
	B^c_{r-|y-x_0|}(y)$ for any $y \in E\cap B_r(x_0)$ and from the co-area formula, we have the following estimate:
	\begin{align}\label{densityEstimate02}
		\int_{E \cap B_r(x_0)}\int_{E \setminus B_r(x_0)} \frac{1}{|x-y|^{N+s}}\,dx\,dy  &\leq \int_{E \cap B_r(x_0)}\int_{B^c_{r-|y-x_0|}(y)} \frac{1}{|x-y|^{N+s}}\,dx\,dy \nonumber\\
		&\leq \frac{|\partial B_1|}{s}\int_{E \cap B_r(x_0)} \frac{1}{(r-|y-x_0|)^s}\,dy \nonumber\\
		&= \frac{|\partial B_1|}{s} \int_{0}^{r}\frac{\capH^{N-1}(E \cap \partial B_{\sigma})}{(r-\sigma)^s}\,d\sigma.
	\end{align}
	Now we set $\phi(r) \coloneqq | E \cap B_r(x_0) |$ for any $r>0$ and we have that, for a.e. $r>0$, $\phi^{\prime}(r) = \capH^{N-1}(E \cap \partial B_{r})$. Thus we obtain from \eqref{densityEstimate01} and \eqref{densityEstimate02}, that
	\begin{align}\label{densityEstimate03}
		C(N,s)\,\phi(r)^{\frac{N-s}{N}} \leq \frac{2|\partial B_1|}{s} \int_{0}^{r}\frac{\phi^{\prime}(\sigma)}{(r-\sigma)^s}\,d\sigma + \frac{2\,E_{s,g}[m]}{m}\,\phi(r)
	\end{align}
	for any small $r>0$ where $C(N,s) \coloneqq |B_1|^{-\frac{N-s}{N}}\,P_s(B_1)$. Now we show that $m^{-1}E_{s,g}[m]$ is bounded by the constant independent of $m \geq 1$. Indeed, from the definition of $E_{s,g}[m]$ and by changing the variable $x \mapsto r_m\,x$, we first have that 
	\begin{align}\label{estimateMinimumUpperBound}
		E_{s,g}[m] \leq \capE_{s,g}(B_{r_m}) &= P_s(B_{r_m}) + V_g(B_{r_m}) \nonumber\\
		&\leq \left(\frac{m}{|B_1|}\right)^{\frac{N-s}{N}}P_s(B_1) + \left(\frac{m}{|B_1|}\right)^{2}\,2\int_{B_1}\int_{B_1}g(r_m(x-y))\,dx\,dy,
	\end{align}
	where $r_m>0$ is the constant satisfying $|B_{r_m}|=m$. Moreover, from the assumptions on $g$ and by changing the variable again, we have that
	\begin{align}\label{estimateMinimumUpperBound02}
		\int_{B_1}\int_{B_1}g(r_m(x-y))\,dx\,dy &\leq \int_{B_1(0)}\sup_{y\in\mathR^N}\int_{B_1(0)}g(r_m\,(x-y))\,dx \,dy \nonumber\\
		&= |B_1|\,r_m^{-N}\int_{B_{r_m}(0)}g(x)\,dx \leq |B_1|^2\|g\|_{L^1(\mathR^N)}\,m^{-1} .
	\end{align}
	Thus, from \eqref{estimateMinimumUpperBound}, \eqref{estimateMinimumUpperBound02} and the assumption that $m \geq 1$, we obtain
	\begin{equation}\nonumber
		m^{-1}E_{s,g}[m] \leq \frac{P_s(B_1)}{|B_1|^{\frac{N-s}{N}}}m^{-\frac{s}{N}} + 2\|g\|_{L^1(\mathR^N)} \leq \frac{P_s(B_1)}{|B_1|^{\frac{N-s}{N}}} + 2\|g\|_{L^1(\mathR^N)} \eqqcolon \tilde{C}(N,s,g)
	\end{equation}
	and this completes the proof of the claim. Since $\phi$ is non-decreasing and $\phi(r) \leq |B_1|\,r^N$ for any $r>0$, we have that 
	\begin{equation}\nonumber
		4\,\tilde{C}(N,s,g)\,\phi(r) \leq C(N,s)\,\phi(r)^{\frac{N-s}{N}}, \quad r_0 \coloneqq \left(\frac{P_s(B_1)}{4\tilde{C}(N,s,g)|B_1|}\right)^{\frac{1}{s}}
	\end{equation}
	for any $r \in (0,\,r_0]$. notice that $r_0$ is independent of $m$. Then integrating the both side of \eqref{densityEstimate03} over $r \in [0,\,r']$ for any $r' \in (0,\,r_0]$, we obtain that
	\begin{equation}\nonumber
		\frac{C(N,s)}{2}\,\int_{0}^{r'}\phi(r)^{\frac{N-s}{N}}\,dr \leq \frac{2|\partial B_1|}{s} \int_{0}^{r'}\int_{0}^{r}\frac{\phi^{\prime}(\sigma)}{(r-\sigma)^s}\,d\sigma \,dr.
	\end{equation}
	By changing the order of the integral, we have that
	\begin{align}
		C'(N,s)\,\int_{0}^{r'}\phi(r)^{\frac{N-s}{N}}\,dr &\leq \int_{0}^{r'}\int_{\sigma}^{r'}\frac{\phi^{\prime}(\sigma)}{(r-\sigma)^s} \,dr\,d\sigma \nonumber\\
		&= \frac{1}{1-s}\int_{0}^{r'}\phi^{\prime}(r)\,(r'-r)^{1-s}\,dr \nonumber\\
		&\leq \frac{(r')^{1-s}}{1-s}\phi(r') \nonumber
	\end{align}
	for any $r' \in (0,\,r_0]$ where we set $C'(N,s) \coloneqq 4^{-1}\,|\partial B_1|^{-1}\,s\,C(N,s)$. Now in order to prove the uniform density estimate, we suppose by contradiction that there exists a constant $r_1 \in (0,\,r_0]$ such that 
	\begin{equation}\nonumber
		|E \cap B_{r_1}(x_0)| \leq c_0^{-\frac{N}{s}}\,r_1^N, \quad c_0 \coloneqq \frac{s(1-s)\,P_s(B_1)}{16|\partial B_1|\,|B_1|^{\frac{N-s}{N}}}.
	\end{equation}
	Then by applying the same argument in \cite[Lemma 3.1]{FFMMM}, we can obtain $|E \cap B_{\frac{r_1}{2}}(x_0)| = 0$, which is a contradiction to the choice of $x_0$. Notice that the constants $c_0$ and $r_0$ are independent of $E$, $x_0$, and $r$. 
\end{proof}

\section{Existence of minimizers for $\capE_{s,g}$ for any volumes}\label{sectionExisMiniFastDecayGAnyVol}
In this section, we prove Theorem \ref{theoremExistMiniAnyVolumeFasterDecay}, namely, the existence of minimizers of the functional $\capE_{s,g}$ for any volume $m>0$ under the assumption that the kernel $g$ of the Riesz potential decays faster than the kernel of the fractional perimeter $P_s$, namely, $x \mapsto |x|^{-(N+s)}$. 

Before proving Theorem \ref{theoremExistMiniAnyVolumeFasterDecay}, we show an auxiliary lemma, which states the existence of minimizers in a given bounded set of $\capE_{s,g}$ with volume $m>0$.
\begin{lemma}\label{lemmaExistMiniInBoundedSet}
	Let $m>0$. Assume that the kernel $g: \mathR^N \setminus \{0\} \to \mathR$ is in $L^1_{loc}(\mathR^N)$. Then, for any $R \geq m^{\frac{1}{N}}$, there exists a minimizer $E^R$ of $\capE_{s,g}$ with $|E|=m$ such that $E^R \subset Q_R$ where $Q_R \subset \mathR^N$ is the cube of size $R$.
\end{lemma}
\begin{remark}
	We can observe that, if $E_n$ is a minimizer of $\capE_{s,g}$ among subsets of $Q_n$ with $|E|=m$ for every $n$ and $\chi_{E_n} \to \chi_{E_{\infty}}$ as $n\to\infty$ in $L^1_{loc}$ with $|E_{\infty}|=m$, then it follows that $E_{\infty}$ is a minimizer of $\capE_{s,g}$ among sets of volume $m$. Indeed, we take any $F \subset \mathR^N$ with $|F|=m$ and set $\mu_n(F) \coloneqq \frac{m}{m-|F \setminus Q_{n-1}|} \geq 1$ for any large $n\in\mathN$. Then, we can readily see that
	\begin{equation}\nonumber
		\mu_n(F) \xrightarrow[n\to\infty]{} 1, \quad |\mu_n(F)\,(F \cap Q_{n-1})| = m, \quad \mu_n(F)\,(F \cap Q_{n-1}) \subset Q_n \quad 
		\text{for large $n\in\mathN$}
	\end{equation}
	Thus, by choosing the set $\mu_n(F)\,(F \cap Q_{n-1})$ as a competitor against the minimizer $E_n$ and from Lemma \ref{lemmaScalingEnergy} and Proposition \ref{propositionIntersectionConvexSmaller}, we have
	\begin{equation}\nonumber
		\capE_{s,g}(E_n) \leq \mu_n(F)^{2N}\capE_{s,g}(F \cap Q_{n-1}) \leq \mu_n(F)^{2N} \capE_{s,g}(F).
	\end{equation}
	Hence, from the lower semi-continuity, we obtain that the claim is valid.
\end{remark}
\begin{proof}
	The proof is followed by the direct method of the calculus of variations and the compact embedding $W^{s,1} \hookrightarrow L^1$ for $s \in (0,\,1)$ in a bounded set (see \cite{dNPV} for the compactness). Let $m>0$ and $R \geq m^{\frac{1}{N}}$ and let $\{E^R_n\}_{n\in\mathN}$ be a minimizing sequence of $\capE_{s,g}$ with $|E^R_n|=m$ and $E^R_n \subset Q_R$. Then, by definition, we have 
	\begin{equation}\nonumber
		\sup_{n\in\mathN} P_s(E^R_n; Q_R) = \sup_{n\in\mathN} P_s(E^R_n; Q_R) \leq \inf\left\{ \capE_{s,g}(E) \mid \text{$E \subset Q_R$, $|E|=m$} \right\} < \infty
	\end{equation}
	and thus, from the compactness and the uniformly boundedness of $\{E^R_n\}_{n\in\mathN}$, we can choose a set $G^R \subset Q_R$ such that, up to extracting a subsequence, $\chi_{E^R_n} \to \chi_{G^R}$ in $L^1(Q_R)$. Moreover, we obtain that $|G^R| = \lim_{n \to \infty}|E^R_n| = m$ and thus, $G^R$ with $|G^R|=m$. Finally, from the lower semi-continuous of the fractional perimeter and Fatou's lemma, we obtain that 
	\begin{equation}\nonumber
		\capE_{s,g}(G^R) \leq \liminf_{n\to\infty} P_s(E^R_n) + \liminf_{n\to\infty} V_g(E^R) = \inf\left\{ \capE_{s,g}(E) \mid \text{$E \subset Q_R$, $|E|=m$} \right\}
	\end{equation}
	with $|G^R| = m$. This implies that $G^R$ is a minimizer.
\end{proof}

Now we are prepared to show Theorem \ref{theoremExistMiniAnyVolumeFasterDecay}. The idea to prove the existence is based on the argument by Di Castro, et al. in \cite{dCNRV} (see also \cite{GoNo, CeNo}). As we mentioned in the introduction, the idea was originally inspired by the so-called ``concentration-compactness" principle introduced by P.~L. Lions in \cite{Lions01, Lions02}. When one studies the variational problems in unbounded domain, the possible loss of compactness may occur from the vanishing or splitting into many pieces of minimizing sequences. 

Although the proof of this method may be technical, we briefly explain the strategy of it in the following; when we obtain the existence of minimizers of the minimization problems of isoperimetric type, we usually apply the direct method in the calculus of variations. More precisely, we first take any minimizing sequence; then we try to construct another sequence of the minimizing sequence in such a way that the new elements are uniformly bounded and the energy of the new sequence is smaller than that of the original sequence (one may often refer to this procedure as ``truncation"); thus, by some compactness, we can extract a convergent subsequence in proper topology; finally, we may conclude that, by lower semi-continuity, the limit of the subsequence should be a minimizer as desired. 

Unfortunately, in our problem, we might not be able to easily construct another sequence, which satisfies ``good" properties we want, from the original minimizing sequence. One possible reason is as follows; as is well-known, the fractional perimeter $P_s$ behaves like an attracting term, while the Riesz potential associated with the kernel $g$ could disaggregate minimizers into many different components. Moreover, in general, as the volume of a minimizer gets larger, the effect that separates minimizers into pieces from the Riesz potential may get stronger. However it is not obvious whether or not the nonlocal perimeter term can overcome such an effect from the Riesz potential because we cannot easily capture the precise behavior of a general kernel $g$. Therefore, we select the following strategy to handle the problems: first, taking any minimizing sequence $\{E_n\}_{n}$ of $\capE_{s,g}$ with $|E_n|=m>0$, we decompose each element $E_n$ into many pieces with the cubes $\{Q_n^i\}_{i}$ in such a way that each piece has non-negligible volumes. Then we ``properly" collect all the components $\{E_n \cap Q_n^i\}_{i}$ of $E_n$ such that $\dist(Q_n^i,\,Q_n^j) \to c^{ij} <\infty$ as $n\to\infty$ for $i \neq j$ (the case that $c^{ij}=\infty$ for $i \neq j$ is called the ``dichotomy" in the sense of Lions'). Thanks to the uniformly boundedness of $\{P_s(E_n)\}_{n}$ and the isoperimetric inequality of $P_s$, we can obtain a sequence of the limit sets $\{G^i\}_{i}$ of the components of $E_n$ that we have ``properly" collected such that $\{G^i\}_{i}$ is the collection with $c^{ij} = \infty$ for any $i \neq j$. Now we need to show that the amount of the volume of $\{G^i\}_{i}$ is equal to $m$ (this means that we exclude the ``vanishing phenomena" in the sense of Lions'). Once we have shown that $\sum_{i}|G^i|=m$, the faster decay of the kernel $g$ in Riesz potential enables us to prove that the only one element in $\{G_i\}_i$ should be the true minimizer of $\capE_{s,g}$ among sets of volume $m$. 

\begin{proof}[Proof of Theorem \ref{theoremExistMiniAnyVolumeFasterDecay}]
	Let $m>0$ be any number and let $\{E_n\}_{n\in\mathN}$ be a sequence of minimizers of $\capE_{s,g}$ with $|E_n|=m$ and $E_n \subset Q_n$. Notice that, as we see in Lemma \ref{lemmaExistMiniInBoundedSet}, the existence of the minimizer is guaranteed. We first decompose $Q_n$ into the unit cubes and denote by $\{Q_n^i\}_{i=1}^{I_n}$, where we choose a number $I_n \in \{1,\cdots,\,n^N\}$ in such a way that $|E_n \cap Q_n^i|>0$ for any $i \in \{1,\cdots,\,I_n\}$. We set $x_n^i \coloneqq |E_n \cap Q_n^i|$ and, since $E_n \subset Q_n$ for any $n\in\mathN$, we have that
	\begin{equation}\label{keyTechnical01}
		\sum_{i=1}^{I_n} x_n^i = |E_n| = m
	\end{equation} 
	for any $n\in\mathN$. Since $E_n$ is a minimizer with $|E_n|=m$ for any $n$, we can choose a ball with the volume $m$ as a competitor and then, from the local integrability of $g$, have
	\begin{equation}\label{uniformBoundednessNonlocPeri}
		\sup_{n\in\mathN}P_s(E_n) \leq P_s(B_m) + V_g(B_m) \leq \left(\frac{m}{|B_1|}\right)^{\frac{N-s}{N}}\,P_s(B_1) + m\,\|g\|_{L^1(2B_m)} < \infty
	\end{equation}
	where $B_m$ is the open ball with the volume $m$ for each $m>0$. From \eqref{uniformBoundednessNonlocPeri} and the isoperimetric inequality shown in \cite[Lemma 2.5]{dCNRV}, we obtain
	\begin{equation}\label{keyTechnical02}
		\sum_{i=1}^{I_n} (x_n^i)^{\frac{N-s}{N}} \leq C\sum_{i=1}^{I_n}P_s(E_n; Q_n^i) \leq 2CP_s(E_n) \leq C_1 < \infty
	\end{equation}
	for any $n\in\mathN$, where $C$ and $C_1$ are the positive constants independent of $n$. Up to reordering the cubes $\{Q_n^i\}_{i}$, we may assume that $\{x_n^i\}_{i}$ is a non-increasing sequence for any $n\in\mathN$. Thus, applying the technical result shown in \cite[Lemma 4.2]{GoNo} or \cite[Lemma 7.4]{dCNRV} with \eqref{keyTechnical01} and \eqref{keyTechnical02}, we obtain that
	\begin{equation}\label{keyTechnical03}
		\sum_{i=k+1}^{\infty} x_n^i \leq \frac{C_2}{k^{\frac{s}{N}}} 
	\end{equation}
	for any $k\in\mathN$, where we set $x_n^i \coloneqq 0$ for any $i > I_n$ and $C_2$ is the positive constant independent of $n$ and $k$. 
 
	Hence, by using the diagonal argument, we have that, up to extracting a subsequence, $x_n^i \to \alpha^i \in [0,\,m]$ as $n \to \infty$ for every $i\in\mathN$ and, from \eqref{keyTechnical01} and \eqref{keyTechnical03}, 
	\begin{equation}\label{identityLimitMeasu}
		\sum_{i=1}^{\infty} \alpha^i = m.
	\end{equation}
	Now we fix the centre of the cube $z_n^i \in Q_n^i$ for each $i$ and $n$. Up to extracting a further subsequence, we may assume that $|z_n^i - z_j^i| \to c^{ij} \in [0,\,\infty]$ as $n \to \infty$ for each $i,\,j \in \mathN$ and, since we have, from \eqref{uniformBoundednessNonlocPeri}, the uniform bound of the sequence $\{P_s(E_n-z_n^i)\}_{n\in\mathN}$ and its upper-bound is independent of $i$, there exists a measurable set $G^i \subset \mathR^N$ such that, up to a subsequence, 
	\begin{equation}\nonumber
		\chi_{E_n-z_n^i} \xrightarrow[n \to \infty]{} \chi_{G^i} \quad \text{in $L^1_{loc}$-topology}.
	\end{equation}
	We define the relation $i \sim j$ for every $i,\,j\in\mathN$ as $c^{ij} < \infty$ and we denote by $[i]$ the equivalent class of $i$. Moreover, we define the set of the equivalent class by $\capI$. Then, in the following, we show a sort of lower semi-continuity, More precisely,
	\begin{equation}\label{lowerSemicontiConcentration}
		\sum_{[i] \in \capI}P_s(G^i) \leq \liminf_{n \to \infty}P_s(E_n), \quad \sum_{[i] \in \capI}V_g(G^i) \leq \liminf_{n \to \infty}V_g(E_n).
	\end{equation}
	Indeed, we first fix $M \in \mathN$ and $R>0$ and we take the equivalent classes $i_1,\cdots,\,i_M$. Notice that, if $p \neq q$, then $|z_n^{i_p} - z_n^{i_q}| \to \infty$ as $n\to\infty$ and thus we have that $\{z_n^{i_p}+Q_R\}_{p}$ are disjoint sets for large $n$ and 
	\begin{equation}\nonumber
		\int_{z_n^{i_p} + Q_R}\int_{z_n^{i_q} + Q_R} \frac{1}{|x-y|^{N+s}}\,dx\,dy \xrightarrow[n \to \infty]{} 0
	\end{equation} 
	where $Q_R$ is the cube of side $R$. We recall the inequality of the nonlocal perimeter;
	\begin{equation}\nonumber
		P_s(E; A) + P_s(E; B) \leq P_s(E; A \sqcup B) + 2\int_{A}\int_{B}\frac{dx\,dy}{|x-y|^{N+s}}
	\end{equation}
	for any measurable disjoint sets $A,\,B\subset \mathR^N$. As a consequence, from the lower semi-continuity of $P_s$, we obtain
	\begin{align}
		\sum_{p=1}^{M} P_s(G^{i_p} ; Q_R) &\leq \liminf_{n \to \infty}	\sum_{p=1}^{M} P_s(E_n-z_n^{i_p} ; Q_R) \nonumber\\
		&= \liminf_{n \to \infty}	\sum_{p=1}^{M} P_s(E_n ; z_n^{i_p}+Q_R) \nonumber\\
		&\leq \liminf_{n \to \infty} P_s\left(E_n; \bigcup_{p=1}\left(z_n^{i_p}+Q_R\right) \right) \nonumber\\
		&\qquad + \liminf_{n \to \infty}2\sum_{p \neq q} \int_{z_n^{i_p} + Q_R}\int_{z_n^{i_q} + Q_R} \frac{dx\,dy}{|x-y|^{N+s}} \nonumber\\
		&\leq \liminf_{n \to \infty} P_s(E_n). \nonumber
	\end{align} 
	Letting $R \to \infty$ and then $M \to \infty$, we obtain the first claim of \eqref{lowerSemicontiConcentration}. For the second claim, we again take any $M\in\mathN$ and $R>0$. We recall the identity
	\begin{equation}\nonumber
		V_g(A) + V_g(B) = V_g(A \sqcup B) - 2\int_{A}\int_{B}g(x-y)\,dx\,dy 
	\end{equation}
	for any measurable disjoint set $A,\,B \subset \mathR^N$. Then, in the same way as we have observed in the first claim, we have, from Fatou's lemma and the non-negativity of $g$, that 
	\begin{align}
		\sum_{p=1}^{M} V_g(G^{i_p} \cap Q_R) &\leq \liminf_{n \to \infty} \sum_{p=1}^{M} V_g \left( \left(E_n-z_n^{i_p}\right) \cap Q_R \right) \nonumber\\
		&= \liminf_{n \to \infty} \sum_{p=1}^{M} V_g \left(E_n \cap \left(z_n^{i_p}+ Q_R\right) \right) \nonumber\\
		&\leq \liminf_{n \to \infty} V_g\left(E_n \cap \bigcup_{p=1}^{M} \left(z_n^{i_p}+Q_R\right)\right) \nonumber\\
		&\leq \liminf_{n \to \infty} V_g(E_n). \nonumber
	\end{align}  
	Here we have used the fact that the sets $\{z_n^{i_p} + Q_R\}_{p=1}^{M}$ are disjoint if $n$ is sufficiently large from the choice of the points $\{z_n^{i_p}\}_{p=1}^{M}$. Thus, letting $R \to \infty$ and then $M\to\infty$, we obtain the second claim. 
	
	Now we show that 
	\begin{equation}\nonumber
		\sum_{[i] \in \capI} |G^i| = m.
	\end{equation}
	Indeed, from the $L^1_{loc}$-convergence of $\{\chi_{E_n - z_n^i}\}_{n\in\mathN}$ for any $i$, we have that, for any $R>0$ sufficiently large,
	\begin{equation}\label{estimateFromBelowLimitMeasu}
		|G^i| \geq |G^i \cap Q_R| = \lim_{n \to \infty}|(E_n-z_n^i) \cap Q_R|.
	\end{equation}
	If $j\in\mathN$ is such that $j \sim i$ and $c^{ij}<\frac{R}{100}$, then we have that $Q_n^i - z_n^i \subset Q_R$ for large $R>0$ and all $n$. Thus, from \eqref{estimateFromBelowLimitMeasu}, it follows
	\begin{align}\label{estimateParticleVolume}
		|(E_n-z_n^i) \cap Q_R| &= \sum_{i=1}^{I_n} |(E_n-z_n^i) \cap Q_R \cap \left(Q_n^i - z_n^i\right)| \nonumber\\
		&\geq \sum_{j:\,c^{ij}<\frac{R}{100}}|(E_n-z_n^i) \cap Q_R \cap \left(Q_n^i - z_n^i\right)| \nonumber\\
		&= \sum_{j:\,c^{ij}<\frac{R}{100}}|E_n \cap Q_n^i|
	\end{align} 
	for all $n$ and large $R>0$. Therefore, combining \eqref{estimateParticleVolume} with \eqref{estimateFromBelowLimitMeasu}, we obtain
	\begin{equation}\nonumber
		|G^i| \geq \sum_{j:\,c^{ij}<\frac{R}{100}} \alpha^i
	\end{equation}
	and, letting $R \to \infty$, we have
	\begin{equation}\nonumber
		|G^i| \geq \sum_{j: \, c^{ij}<\infty} \alpha^i = \sum_{j \in [i]} \alpha^i.
	\end{equation}
	Hence, recalling \eqref{identityLimitMeasu}, we have 
	\begin{equation}\label{keyVolumeInequality}
		\sum_{[i] \in \capI} |G^i| \geq \sum_{[i] \in \capI}\sum_{j \in [i]} \alpha^i = m.
	\end{equation}
	For the other inequality, we can easily obtain from the choice of $\{G^i\}_{i}$ in the following manner; for any $M\in\mathN$ and $R>0$, we take the equivalent classes $i_1,\cdots,\,i_M$ and then have that
	\begin{align}\label{keyVolumeEstimateOtherDirec}
		\sum_{p=1}^{M} |G^{i_p} \cap Q_R| &= \lim_{n \to \infty} \sum_{p=1}^{M} \left|\left(E_n - z_n^{i_p}\right) \cap Q_R \right| \nonumber\\
		&=  \lim_{n \to \infty} \sum_{p=1}^{M} \left|E_n \cap \left(z_n^{i_p} + Q_R\right) \right|.
	\end{align}
	Recalling the condition that $|z_n^{i_p} - z_n^{i_q}| \to \infty$ as $n \to \infty$ if $p \neq q$, we have that, for sufficiently large $n\in\mathN$, $\left(z_n^{i_p} + Q_R\right) \cap \left(z_n^{i_q} + Q_R\right) = \emptyset$ for any $p \neq q$. From \eqref{keyVolumeEstimateOtherDirec}, we have that
	\begin{equation}\nonumber
		\sum_{p=1}^{M} |G^{i_p} \cap Q_R| =  \lim_{n \to \infty}  \left|E_n \cap \bigcup_{p=1}^{M}\left(z_n^{i_p} + Q_R\right) \right| \leq m
	\end{equation}
	and thus, letting $R \to \infty$ and then $M \to \infty$, we obtain that 
	\begin{equation}\nonumber
		\sum_{[i] \in \capI} |G^{i}| = \sum_{p=1}^{\infty} |G^{i_p}| \leq m.
	\end{equation}
	This completes the proof of the claim. Taking into account all the above arguments, we obtain the existence of sets $\{G^i\}_{[i]\in\capI}$ satisfying the properties that
	\begin{equation}\label{keyPropertylowSemicontiVolumeEqualParticles}
		\sum_{[i] \in \capI} \capE_{s,g}(G^i) \leq \liminf_{n \to \infty} \capE_{s,g}(E_n), \quad \sum_{[i] \in \capI} |G^{i}| = m. 
	\end{equation}

	Now we claim that each particle $G^i$ for $[i] \in \capI$ is a minimizer of $\capE_{s,g}$ among sets of volume of $|G^i|$. Moreover, we show that $G^i$ is bounded for each $[i] \in \capI$. Indeed, we first recall the definition of $E_{s,g}$, which says that 
	\begin{equation}\nonumber
		E_{s,g}[m] \coloneqq \inf\left\{ \capE_{s,g}(E) \mid |E|=m \right\}
	\end{equation} 
	for any $m>0$, and the sub-additivity result of the function $m \mapsto E_{s,g}[m]$ as is shown in Lemma \ref{lemmaSubadditivityEnergy}. Notice that, in this theorem, we impose assumption $(\mathrm{g}3)$ as we show in Section \ref{sectionPreliminary}, which is stronger than $(\mathrm{g}4)$. Thus, we can apply Lemma \ref{lemmaSubadditivityEnergy} to the case in the present proof. Then, from \eqref{keyPropertylowSemicontiVolumeEqualParticles}, we have that
	\begin{align}\label{energyMinimalityEachPrticle}
		\sum_{p=1}^{M} \left(\capE_{s,g}(G^{i_p}) - E_{s,g}[|G^{i_p}|] \right) &\leq E_{s,g}[m] -  \sum_{p=1}^{M} E_{s,g}[|G^{i_p}|] \nonumber\\
		&\leq E_{s,g}\left[ \sum_{p=M+1}^{\infty}|G^{i_p}| \right] + E_{s,g}\left[ \sum_{p=1}^{M}|G^{i_p}| \right] -  \sum_{p=1}^{M} E_{s,g}[|G^{i_p}|] \nonumber\\
		&\leq E_{s,g}\left[ \sum_{p=M+1}^{\infty}|G^{i_p}| \right]
	\end{align}
	for any $M\in\mathN$. We can observe that $E_{s,g}[m] \to E_{s,g}[0] = 0$ as $m \to 0$ because $E_{s,g}[m]$ can be bounded by the quantity $C_1\,m^{\frac{N-s}{N}} + C_2\,m$ for small $m>0$, where $C_1$ and $C_2$ are the constants depending only on $N$, $s$, and $g$. Hence, letting $M \to \infty$ in \eqref{energyMinimalityEachPrticle}, we obtain that 
	\begin{equation}\nonumber
		\sum_{[i] \in \capI} \left(\capE_{s,g}(G^{i_p}) - E_{s,g}[|G^{i_p}|] \right) =  \sum_{p=1}^{\infty} \left(\capE_{s,g}(G^{i_p}) - E_{s,g}[|G^{i_p}|] \right) \leq 0
	\end{equation}
	and, from the fact that each term of the series is non-negative, we conclude that each term of the series is equal to zero. This implies that, for every $[i]\in\capI$, $G^i$ is a minimizer of $\capE_{s,g}$ among sets of volume of $|G^i|$. To see the boundedness of $\{G^i\}_{[i]\in\capI}$, it is sufficient to apply Lemma \ref{lemmaBoundednessMinimizers} to $G^i$ for each $[i]\in\capI$.
	
	Now we show that the set $\capI$ of the equivalent classes is actually a finite set. Indeed, we first set $m^{i_p} \coloneqq |G^{i_p}|$ for any $p \in \mathN$ and, since $\sum_{p=1}^{\infty} m^{i_p} = m$, we can observe that $m^{i_p} \to 0$ as $p \to \infty$ and, moreover, $\mu_{\ell} \coloneqq \sum_{p=\ell+1}^{\infty} m^{i_p} \to 0$ as $\ell \to \infty$. Then, we can choose $\widetilde{p} \in \mathN$ such that $m^{i_{\widetilde{p}}} \geq \frac{m}{2^{\widetilde{p}+1}}$. Now using the sets $\{G^{i_p}\}_{p=1}^{\infty}$, we construct a new family of sets $\{\widetilde{G}^{i_p}\}_{p=1}^{H}$ for some $H \in \mathN$, depending only on $N$, $s$, and $m$, in the following manner; we choose $H \in \mathN$ so large that $H \geq \widetilde{p}$ and set $\widetilde{G}^{i_p} \coloneqq G^{i_p}$ for any $p \in \{1,\cdots ,\,H\}$ with $p \neq \widetilde{p}$ and $\widetilde{G}^{i_{\widetilde{p}}} \coloneqq \lambda\,G^{i_{\widetilde{p}}}$ where $\lambda^{N} \coloneqq \frac{m^{i_{\widetilde{p}}} + \mu_{H}}{m^{i_{\widetilde{p}}}}$. Then, we have the volume identity that
	\begin{equation}\label{volumeIdentityReorganized}
		\sum_{p=1}^{H} \left| \widetilde{G}^{i_p} \right| = \sum_{p=1, \,p\neq\widetilde{p}}^{H} \left| G^{i_p} \right| + \lambda^N\, |G^{i_{\widetilde{p}}}| = \sum_{p=1,\,p\neq\widetilde{p}}^{H} m^{i_p} + m^{i_{\widetilde{p}}} + \mu_{H} = m.
	\end{equation}
	Now we compute the energy for $\{\widetilde{G}^{i_p}\}_{p=1}^{H}$ as follows to show that the total energy of each elements of $\{\widetilde{G}^{i_p}\}_{p=1}^{H}$ is more efficient than that of $\{G^{i_p}\}_{p=1}^{\infty}$; from the definition of $\lambda \geq 1$ and Lemma \ref{lemmaScalingEnergy}, we have that
	\begin{align}\label{comparisonNewSetsEnergy}
		\sum_{p=1}^{H} \capE_{s,g}(\widetilde{G}^{i_p}) &\leq \sum_{p=1,\,p\neq\widetilde{p}}^{H} \capE_{s,g}(G^{i_p}) + \lambda^{2N}\,\capE_{s,g}(G^{i_{\widetilde{p}}}) \nonumber\\
		&= \sum_{p=1}^{\infty} \capE_{s,g}(G^{i_p}) + \left(\lambda^{2N} - 1\right)\,\capE_{s,g}(G^{i_{\widetilde{p}}}) - \sum_{p=H+1}^{\infty}\capE_{s,g}(G^{i_p}) \nonumber\\
		&\leq \sum_{[i] \in \capI} \capE_{s,g}(G^{i_p}) + \frac{2^{\widetilde{p}+1}\,E_{s,g}[m]}{m}\,\mu_{H} - \sum_{p=H+1}^{\infty} P_s(G^{i_p}).
	\end{align}
	Here, in the last inequality, we have also used \eqref{keyPropertylowSemicontiVolumeEqualParticles}. From the isoperimetric inequality of $P_s$ and \eqref{comparisonNewSetsEnergy}, we further obtain that
	\begin{align}
		\sum_{p=1}^{H} \capE_{s,g}(\widetilde{G}^{i_p}) &\leq \sum_{[i] \in \capI} \capE_{s,g}(G^i) + \frac{2^{\widetilde{p}+1}\,E_{s,g}[m]}{m}\, \mu_{H} - C\sum_{p=H+1}^{\infty} \left(m^{i_p}\right)^{\frac{N-s}{N}}  \nonumber\\
		&\leq \sum_{[i] \in \capI} \capE_{s,g}(G^i) + \frac{2^{\widetilde{p}+1}\,E_{s,g}[m]}{m}\, \mu_{H} - C \left(\sum_{p=H+1}^{\infty} m^{i_p}\right)^{\frac{N-s}{N}} \nonumber\\
		&= \sum_{[i] \in \capI} \capE_{s,g}(G^i) + \frac{2^{\widetilde{p}+1}\,E_{s,g}[m]}{m}\, \mu_{H} - C \left( \mu_{H} \right)^{\frac{N-s}{N}}. \nonumber
	\end{align}
	Taking the number $H$ so large that $H \geq \widetilde{p}$ and 
	\begin{equation}\nonumber
		\frac{2^{\widetilde{p}+1}\,E_{s,g}[m]}{m}\, \mu_{H} - C \left( \mu_{H} \right)^{\frac{N-s}{N}} \leq 0,
	\end{equation}
	then we finally obtain that 
	\begin{equation}\label{reductionFiniteElementsMinimizer}
		\sum_{p=1}^{H} \capE_{s,g}(\widetilde{G}^{i_p}) \leq \sum_{[i] \in \capI} \capE_{s,g}(G^i) \leq \liminf_{n \to \infty}\capE_{s,g}(E_n).
	\end{equation}
	This completes the proof of the claim.
	
	Finally, we show that there exists one number $i_0 \in \mathN$ such that $|G^i|=0$ for any $[i] \in \capI$ with $i \neq i_0$, using the assumption $(\mathrm{g}3)$. This proves the existence of minimizers of $\capE_{s,g}$, namely, Theorem \ref{theoremExistMiniAnyVolumeFasterDecay}. Indeed, if the claim is true, then from \eqref{lowerSemicontiConcentration} and \eqref{keyVolumeInequality}, we obtain
	\begin{equation}\nonumber
		\capE_{s,g}(G^{i_0}) = P_s(G^{i_0}) + V_g(G^{i_0}) = \sum_{[i] \in \capI}\left(P_s(G^i) + V_g(G^i) \right) \leq \liminf_{n \to \infty} \capE_{s,g}(E_n) = E_{s,g}[m]
	\end{equation}
	and 
	\begin{equation}\nonumber
		|G^{i_0}| = \sum_{[i]\in\capI}|G^i| =m.
	\end{equation}
	Therefore, $G^{i_0}$ is a minimizer of $\capE_{s,g}$ with $|G^{i_0}| = m$ for any $m>0$. 
	
	From \eqref{keyVolumeInequality}, there exists at least one number $p'\in\mathN$ such that $|G^{i_{p'}}| > 0$. Then we claim that, if $q \neq p'$, then it holds that $|G^{i_q}|=0$. Indeed, from the previous claim, we can restrict ourselves to consider a finite number of sets $\{\widetilde{G}^{i_p}\}_{p=1}^{H}$, which satisfies \eqref{reductionFiniteElementsMinimizer} and $\sum_{p=1}^{H}|\widetilde{G}^{i_p}| = m$, instead of $\{G^i\}_{[i]\in\capI}$. Moreover, we may assume that $H \geq p'$. Since we have shown that the sets $\{\widetilde{G}^{i_p}\}_{p=1}^{H}$ are bounded, we can choose the points $\{z^{i_p}\}_{p=1,\,p \neq p'}^{H}$ such that each set $\widetilde{G}^{i_p} + R\,z^{i_p}$ is far away from the others for large $R>1$. We can thus compute the energy as follows; from the translation invariance of $\capE_{s,g}$, it holds that
	\begin{align}
		\sum_{p=1}^{H} \capE_{s,g}(\widetilde{G}^{i_p}) &= \sum_{p=1\,p \neq p',q}^{H} \capE_{s,g}(\widetilde{G}^{i_p}) + \capE_{s,g}(\widetilde{G}^{i_{p'}}) + \capE_{s,g}(\widetilde{G}^{i_q}) \nonumber\\
		&= \sum_{p=1\,p \neq p',q}^{H} \capE_{s,g}(\widetilde{G}^{i_p}) + \capE_{s,g}(\widetilde{G}^{i_{p'}}) + \capE_{s,g}(\widetilde{G}^{i_q} + R\,z^{i_q}) \nonumber\\
		&= \sum_{p=1\,p \neq p',q}^{H} \capE_{s,g}(\widetilde{G}^{i_p}) + \capE_{s,g}(\widetilde{G}^{i_{p'}} \cup (\widetilde{G}^{i_q} + R\,z^{i_q})) \nonumber\\
		&\qquad + 2\int_{\widetilde{G}^{i_{p'}}} \int_{\widetilde{G}^{i_q} + R\,z^{i_q}} \frac{dx\,dy}{|x-y|^{N+s}} - 2\int_{\widetilde{G}^{i_{p'}}} \int_{\widetilde{G}^{i_q} + R\,z^{i_q}}g(x-y)\,dx\,dy \nonumber
	\end{align}
	for any $q \in \{1,\cdots,\,H\}$ with $q \neq p'$ and sufficiently large $R>1$. Recalling the assumption $(\mathrm{g}3)$ that $g(x) \leq \beta|x|^{-(N+s)}$ for any $|x| \geq R_0$ and some $\beta\in(0,\,1)$, and choosing $R>1$ in such a way that the set $\widetilde{G}^{i_q} + R\,z^{i_q}$ has the distance of more than $R_0$ from $\widetilde{G}^{i_{p'}}$, we obtain that
	\begin{align}\label{keyEstimateReductionOneElement}
		\sum_{p=1}^{H} \capE_{s,g}(\widetilde{G}^{i_p}) &\geq \sum_{p=1,\,p \neq p',\,q}^{H} \capE_{s,g}(\widetilde{G}^{i_p}) + \capE_{s,g}(\widetilde{G}^{i_{p'}} \cup (\widetilde{G}^{i_q} + R\,z^{i_q})) \nonumber\\
		&\qquad + 2(1-\beta)\int_{\widetilde{G}^{i_{p'}}} \int_{\widetilde{G}^{i_q} + R\,z^{i_q}} \frac{dx\,dy}{|x-y|^{N+s}}. 
	\end{align}
	By repeating the same argument finite times for the rest of the sets $\{\widetilde{G}^{i_p}\}_{p=1,\,p \neq p',q}^{H}$ with sufficiently large $R>1$, we obtain the similar inequalities to \eqref{keyEstimateReductionOneElement} and, finally, we can derive the inequality that
	\begin{align}\label{keyEstimateReductionOneElement02}
		\sum_{p=1}^{H} \capE_{s,g}(\widetilde{G}^{i_p}) &\geq \capE_{s,g}\left( \widetilde{G}^{i_{p'}} \cup \bigcup_{p=1,\,p \neq p'}^{H}\left( \widetilde{G}^{i_p} + R\,z^{i_p} \right) \right) \nonumber\\
		&\qquad + 2(1-\beta)\sum_{p=1,\,p \neq p'}^{H} \int_{\widetilde{G}^{i_{p'}}} \int_{\widetilde{G}^{i_p} + R\,z^{i_p}} \frac{dx\,dy}{|x-y|^{N+s}}.
	\end{align}
 	Since $\widetilde{G}^{i_{p'}} \cup \bigcup_{p=1,\,p \neq p'}^{H}\left( \widetilde{G}^{i_p} + R\,z^{i_p} \right)$ are the union of disjoint sets, we have, from \eqref{volumeIdentityReorganized}, that 
 	\begin{equation}\nonumber
 		\left| \widetilde{G}^{i_{p'}} \cup \bigcup_{p=1,\,p \neq p'}^{H}\left( \widetilde{G}^{i_p} + R\,z^{i_p} \right)\right| = \sum_{p=1}^{H} |\widetilde{G}^{i_{p}}| = m.
 	\end{equation}
 	Thus, from \eqref{keyEstimateReductionOneElement02}, we obtain
 	\begin{align}
 		2(1-\beta)&\sum_{p=1,\,p \neq p'}^{H} \int_{\widetilde{G}^{i_{p'}}} \int_{\widetilde{G}^{i_p} + R\,z^{i_p}} \frac{dx\,dy}{|x-y|^{N+s}} + E_{s,g}[m] \nonumber\\
 		&\leq \sum_{p=1,\,p \neq p'}^{H} \int_{\widetilde{G}^{i_{p'}}} \int_{\widetilde{G}^{i_p} + R\,z^{i_p}} \frac{dx\,dy}{|x-y|^{N+s}} + \capE_{s,g}\left( \widetilde{G}^{i_{p'}} \cup \bigcup_{p=1,\,p \neq p'}^{H}\left( \widetilde{G}^{i_p} + R\,z^{i_p} \right) \right) \nonumber\\
 		&\leq  \sum_{p=1}^{H} \capE_{s,g}(\widetilde{G}^{i_p}) \leq E_{s,g}[m] \nonumber
 	\end{align}
 	and it follows that
 	\begin{equation}\nonumber
 		2(1-\beta)\sum_{p=1,\,p \neq p'}^{H} \int_{\widetilde{G}^{i_{p'}}} \int_{\widetilde{G}^{i_p} + R\,z^{i_p}} \frac{dx\,dy}{|x-y|^{N+s}} \leq 0
 	\end{equation}
 	for large $R>1$. Since each term of the sum is non-negative, $\beta < 1$, and $|\widetilde{G}^{i_{p'}}|>0$, we conclude that $|\widetilde{G}^{i_p}| = 0$ for all $p \neq p'$. Therefore, the final claim is valid and this completes the proof of Theorem \ref{theoremExistMiniAnyVolumeFasterDecay}.
\end{proof}

\subsection{Regularity of the boundaries of minimizers}
In this subsection, we consider the regularity of the boundary of a minimizer of $\capE_{s,g}$ under suitable conditions on the kernel $g$. To see this, we recall several results on the regularity of the so-called \textit{almost $s$-fractional minimal surfaces}. The first one is on the $C^1$-regularity of the $s$-fractional almost minimal surfaces, which was shown by Caputo and Guillen in \cite{CaGu}.

\begin{theorem}[\cite{CaGu}] \label{regularityCaputoGuillen}
	Let $s \in (0,\,1)$ and $\delta>0$, and let $\Omega \subset \mathR^N$ be any bounded domain with Lipschitz boundary. Suppose $E$ is a $(P_s,\rho,\delta,)$-minimal in $\Omega$, where $\rho: (0,\,\delta) \to \mathR$ is a non-decreasing and bounded  function with some growth condition. Here we mean by  $(P_s,\rho,\delta,)$-minimal in $B_R$ for some $R>0$ that for any $x_0\in\partial E$, a measurable set $F \subset \mathR^N$, and $0 < r < \min\{ \delta, \, \dist(x_0,\partial B_R)\}$ with $E \bigtriangleup F \subset B_r(x_0)$, we have
	\begin{equation}\label{almostMinimizerDef}
		P_s(E;B_R) \leq P_s(F;B_R) + \rho(r)\,r^{N-s}.
	\end{equation}
	Then $\partial E$ is of class $C^1$ in $B_{\frac{R}{2}}$, except a closed set of $\capH^{N-2}$-dimension.   
\end{theorem}
\begin{remark}
	We remark that we can choose, for instance, the function $r \mapsto C\,r^{\delta}$ with $0< \delta \leq s$ and some constant $C>0$ as the function $\rho$ in Theorem \ref{regularityCaputoGuillen}. Hence, we are allowed to consider the exponent of the growth term in \eqref{almostMinimizerDef} up to $N$.
\end{remark}

Now we recall another result of the regularity of the boundary of minimizers. Namely, we show the improvement of flatness statement proved in \cite[Theorem 3.4]{FFMMM} by using the method developed in \cite{CaGu}. This result implies $C^{1,\alpha}$-regularity of almost minimal surfaces by using a standard argument.
\begin{theorem}[\cite{CaGu, FFMMM}]\label{improvementFlatnessFFMMM}
	Let $s_0 \in (0,\,1)$ and $\Lambda >0$. Then there exist $\tau,\,\eta,\,q \in (0,\,1)$ depending only on $N$, $s_0$, and $\Lambda$ with the following property: we assume that $E$ is a $\Lambda$-minimizer of the $s$-perimeter for some $s \in [s_0,\,1)$ with $x_0 \in \partial E$. Here we mean by $\Lambda$-minimizer that $E$ is a bounded measurable set in $\mathR^N$ satisfying the condition that, for any bounded set $F \subset \mathR^N$,
	\begin{equation}\nonumber
		P_s(E) \leq P_s(F) + \frac{\Lambda}{1-s}|E \Delta F|.
	\end{equation}
	Then, if 
	\begin{equation}\nonumber
		\partial E \cap B_1(x_0) \subset \{y \mid |(y-x_0)\cdot e | < \tau \}
	\end{equation}
	for some $e \in \mathS^{N-1}$, then there exists $e_0 \in \mathS^{N-1}$ such that 
	\begin{equation}\nonumber
		\partial E \cap B_\eta(x_0) \subset \{y \mid |(y-x_0)\cdot e_0 | < q\,\tau,\eta \}.
	\end{equation}
\end{theorem}

Originally, the regularity of nonlocal minimal surfaces was obtained by Caffarelli, Roquejoffre, and Savin in \cite{CRS}, which states that every $s$-minimal surface is locally $C^{1,\alpha}$ except the singular sets of $\capH^{N-2}$-dimension. More importantly, thanks to the result by Savin and Valdinoci in \cite{SaVa}, the singular set of $s$-minimal surfaces has the Hausdorff dimension up to $N-3$. Hence one can have that $s$-minimal surfaces in $\mathR^2$ are fully $C^{1,\alpha}$-regular.

As a consequence of these regularity results, we obtain the regularity of the minimizers of $\capE_{s,g}$. Before stating the regularity result, we reduce the minimization problem $E_{s,g}[m]$ for any $m>0$ to another minimization problem because that reduction may allow us to consider more easily. More precisely, we show that any solutions of the minimization problem $E_{s,g}[m]$ are also the solutions of the unconstrained minimization problem
\begin{equation}\nonumber
	\inf\left\{ \capE_{s,g,\mu_0}(E) \mid \text{$E \subset \mathR^N$: measurable} \right\}
\end{equation}
for some constant $\mu_0>0$ and any $m>0$, where we define $\capE_{s,g,\mu_0}$ as
\begin{equation}\nonumber
	\capE_{s,g,\mu}(F) \coloneqq\capE_{s,g}(F)+ \mu\,\left| |F|-m \right|
\end{equation}
for any $F \subset \mathR^N$ and $\mu>0$.
\begin{proposition}[Reduction to an unconstrained problem]\label{propositionEquivalenceProblem}
	Let $m>0$. Assume that the kernel $g$ satisfies the conditions $(\mathrm{g}1)$ and $(\mathrm{g}2)$. Then there exists a constant $\mu_0=\mu_0(N,s,g,m)>0$ such that, if $E$ is a minimizer of $\capE_{s,g}$ with $|E|=m$, then $E$ is also a minimizer of $\capE_{s,g,\mu}$ among sets in $\mathR^N$ for any $\mu \geq \mu_0$.
\end{proposition}
\begin{proof}
	Suppose by contradiction that, for any $n \in \mathN$, there exist a minimizer $E_n$ of $\capE_{s,g}$ with $|E_n|=m$ and a constant $\mu(n) \geq n$ such that $E_n$ is not a minimizer of $\capE_{s,g,\mu(n)}$. Then, by assumption, we can choose a sequence $\{F_n\}_{n\in\mathN}$ such that 
	\begin{equation}\label{estimateFromSBC}
		\capE_{s,g,\mu(n)}(F_n) < \capE_{s,g,\mu(n)}(E_n)
	\end{equation}
	for any $n \in \mathN$. First of all, we show that $|F_n| \xrightarrow[n\to\infty]{} m$. Indeed, we set $B_m$ as a open ball in $\mathR^N$ whose volume is equal to $m$. Then from \eqref{estimateFromSBC} and the minimality of $E_n$ with $|E_n|=m$ for any $n\in\mathN$, we have that
	\begin{equation}\label{minimalityInequality}
		\capE_{s,g,\mu(n)}(F_n) < \capE_{s,g,\mu(n)}(E_n) = \capE_{s,g}(E_n) = E_{s,g}[m].
	\end{equation}
	Thus, denoting $r_m$ by the radius of the ball $B_m$ and using the change of variables, we obtain
	\begin{equation}\label{convergenceofF_n}
		\mu(n)\,||F_n| - m| <  E_{s,g}[m] < \infty
	\end{equation}
	for any $n\in\mathN$. From the definition of $r_m$, the right-hand side in \eqref{convergenceofF_n} is finite and independent of $n$. Hence, letting $n \to \infty$ in \eqref{convergenceofF_n}, we obtain the claim that $|F_n| \rightarrow m$ as $n\to\infty$. Finally, we derive a contradiction in the following manner. We first define $\widetilde{F}_n$ as $\widetilde{F}_n \coloneqq \lambda_n \, F_n$ where $\lambda_n^N \coloneqq m\,|F_n|^{-1}$ and, by definition, we can observe that $|\widetilde{F}_n| = m$. In the sequel, we may assume that, up to extracting a subsequence, $|F_n| \leq m$ for $n\in\mathN$. Indeed, we suppose by contradiction that, for any subsequence $\{F_{n_k}\}_{k\in\mathN}$ of $\{F_n\}_{n\in\mathN}$, we always have that $|F_{n_k}| > m$ for any $k\in\mathN$. From the continuity of the Lebesgue measure, for each $k\in\mathN$, there exists a constant $R_k>0$ such that $|F_{n_k} \cap B_{R_{n_k}}(0)| = m$ for every $k\in\mathN$. Thus, from the minimality of $E_n$ for any $n\in\mathN$ and Proposition \ref{propositionIntersectionConvexSmaller}, we have the estimate that
	\begin{equation}\nonumber
		\capE_{s,g,\mu(n)}(E_{n_k}) = \capE_{s,g}(E_{n_k}) \leq \capE_{s,g}(F_{n_k} \cap B_{R_{n_k}}(0)) \leq P_s(F_{n_k}) + V_g(F_{n_k}) = \capE_{s,g}(F_{n_k})
	\end{equation}  
	for any $k\in\mathN$, which contradicts the estimate \eqref{estimateFromSBC} since $\capE_{s,g}(F_{n_k}) \leq \capE_{s,g,\mu(n)}(F_{n_k})$ for any $k\in\mathN$. Hence, from \eqref{estimateFromSBC}, the minimality of $E_n$, the assumption that $\lambda_n \geq 1$ for any $n \in \mathN$, and Lemma \ref{lemmaScalingEnergy}, we have
	\begin{equation}\label{secondEstiMinimality}
		\capE_{s,g,\mu(n)}(F_n) < \capE_{s,g}(E_n) \leq \capE_{s,g}(\widetilde{F}_n) \leq \lambda_n^{2N} \capE_{s,g}(F_n).
	\end{equation} 
	From the definition, we notice that $||F_n| - m| = |\lambda_n^{-N} \, m - m| = m\,\lambda_n^{-N} \, |\lambda_n^N - 1|$ for any $n$. Hence, from \eqref{secondEstiMinimality} and dividing the both side of \eqref{secondEstiMinimality} by $||F_n| - m|$, we obtain
	\begin{equation}\label{secondEstiMinimality03}
		\mu(n) \leq m^{-1}\,\lambda_n^N \frac{|\lambda_n^{2N}-1|}{|\lambda_n^N-1|} P_s(F_n) + m^{-1}\,\lambda_n^N \frac{|\lambda_n^{2N}-1|}{|\lambda_n^N-1|} V_g(F_n)
	\end{equation}
	for any $n\in\mathN$. Recalling \eqref{estimateFromSBC} and \eqref{minimalityInequality}, we have that $P_s(F_n) + V_g(F_n) < E_{s,g}[m] < \infty$. Moreover, we observe that $\frac{|\lambda_n^{2N}-1|}{|\lambda_n^N-1|} \leq 2$ for sufficiently large $n \in \mathN$. Therefore, from \eqref{secondEstiMinimality03}, we obtain
	\begin{equation}\label{keyEstimateContradictionArgument}
		\mu(n) \leq 6m^{-1}\,E_{s,g}[m] 
	\end{equation}
	for sufficiently large $n \in \mathN$ and thus obtain a contradiction.
\end{proof}

Now we are prepared to show the regularity of minimizers for $\capE_{s,g}$
\begin{lemma}[Regularity of minimizers]\label{lemmaRegularityMinimizers}
	Let $s \in (0,\,1)$ and let $m>0$. Assume that the kernel $g: \mathR^N \setminus \{0\} \to \mathR$ satisfies $(\mathrm{g}1)$, $(\mathrm{g}2)$, and $(\mathrm{g}3)$. If $E \subset \mathR^N$ is a minimizer of $\capE_{s,g}$ among sets of volume $m$, then $\partial E$ is of class $C^{1,\alpha}$ with $0 < \alpha < 1$, except a closed set of $\capH^{N-3}$-dimension.
\end{lemma}
\begin{proof}
	First of all, from Lemma \ref{lemmaBoundednessMinimizers}, we have the essential boundedness of the minimizer $E \subset \mathR^N$, namely, $E \subset B_{R_1}(0)$ up to negligible sets for some $R_1>0$. Without loss of generality, we may assume that $R_1 \geq R_0$ where $R_0$ is given in assumption $(\mathrm{g}3)$ in Section \ref{sectionPreliminary}. In order to apply the regularity result of Theorem \ref{improvementFlatnessFFMMM} to our case, it is sufficient to show that the set $E$ is $\Lambda$-minimizer in the sense of Theorem \ref{improvementFlatnessFFMMM} for some constant $\Lambda>0$ independent of $E$. From Proposition \ref{propositionEquivalenceProblem}, we know that $E$ with $|E|=m$ is also a solution to the minimization problem
	\begin{equation}\nonumber
		\min\{\capE_{s,g,\mu_0}(E) \mid \text{$E\subset\mathR^N$: measurable}\}.
	\end{equation}
	where $\mu_0>0$ is as in Proposition \ref{propositionEquivalenceProblem} and is independent of $E$. Hence, from the minimality of $E$, we have that
	\begin{align}\label{minimalitySetE}
		\capE_{s,g,\mu_0}(E) \leq \capE_{s,g,\mu_0}(F)
	\end{align}
	for any bounded measurable set $F \subset \mathR^n$. We may assume that $F$ is finite with respect to the $s$-fractional perimeter; otherwise the inequality \eqref{minimalitySetE} is obviously valid. Then from the fact that $|E|=m$, we have
	\begin{align}\label{almostMinimizersEstimate}
		P_s(E) &\leq P_s(F) + V_g(F) - V_g(E) + \mu_0\,||F|-|E|| \nonumber\\
		&\leq P_s(F) + V_g(F) - V_g(E) + \mu_0\,|F \Delta E|. 
	\end{align}
	Regarding the Riesz potential, we can compute the difference $ V_g(F) - V_g(E)$ as follows:
	\begin{align}\label{estimateRiesz}
		| V_g(F) - V_g(E)| &\leq \left| \int_{F}\int_{F \cup E}g(x-y) \,dx\,dy - \int_{E}\int_{F \cup E}g(x-y) \,dx\,dy \right| \nonumber\\
		&\leq 2 \int_{F \Delta E}\int_{F \cup E}g(x-y) \,dx\,dy \nonumber\\
		&\leq 2|F \Delta E| \,\int_{\mathR^N} g(x)\,dx.
	\end{align}
	Note that, from the local integrability of $g$ and assumption $(\mathrm{g}3)$, the kernel $g$ is integrable in $\mathR^N$ and thus, the right-hand side in \eqref{estimateRiesz} is finite. Hence, by combining \eqref{estimateRiesz} with \eqref{almostMinimizersEstimate}, we obtain that 
	\begin{equation}\nonumber
		P_s(E) \leq P_s(F) +  \left( 2\|g\|_{L^1(\mathR^N)} + \mu_0\right)\,|F \Delta E| 
	\end{equation} 
	for any measurable set $F \subset \mathR^N$. Therefore, by employing Theorem \ref{regularityCaputoGuillen} and \ref{improvementFlatnessFFMMM}, we can conclude that the claim is valid.
\end{proof}

\section{Existence of generalized minimizers for $\widetilde{\capE}_{s,g}$ for any volumes}\label{sectionExistGeneMiniAnyVol}
In this section, we prove Theorem \ref{theoremExistGeneralizedMiniAnyVolume}, namely, the existence of generalized minimizers for the generalized functional $\widetilde{\capE}_{s,g}$ given as \eqref{defiGeneralziedFunctional} for any volumes. To see this, we impose slightly more general assumptions on $g$ than we do to prove the existence of minimizers of $\capE_{s,g}$ for any volumes in Section \ref{sectionExisMiniFastDecayGAnyVol}. More precisely, we assume that the kernel $g \in L^1_{loc}(\mathR^N)$ satisfies the assumptions $(\mathrm{g}1)$, $(\mathrm{g}2)$, and $(\mathrm{g}4)$ in Section \ref{sectionPreliminary}.

Before proving the main theorem, we show one lemma, saying that one can modify a ``generalized" minimizing sequence for the generalized functional $\widetilde{\capE}_{s,g}$ into a ``usual" minimizing sequence for the functional $\capE_{s,g}$. More precisely, we prove
\begin{lemma}\label{lemmaModifyMinimizingSeq}
	Let $s\in(0,\,1)$. Assume that the kernel $g : \mathR^N \setminus \{0\} \to \mathR$ satisfies the assumptions $(\mathrm{g}1)$, $(\mathrm{g}2)$, and $(\mathrm{g}4)$. Then, for any $m>0$, it follows that
	\begin{equation}\nonumber
		\inf\left\{\capE_{s,g}(E) \mid |E|=m \right\} = \inf\left\{\widetilde{\capE}_{s,g}(\{E^k\}_k) \mid \sum_{k=1}^{\infty} |E^k| = m \right\}.
	\end{equation}
\end{lemma}
\begin{proof}
	The proof of this lemma proceeds in a similar manner to the method in the proof of Theorem \ref{theoremExistMiniAnyVolumeFasterDecay}; however, it seems a little technical and thus we do not omit the detail. 
	
	First of all, we can readily observe that the inequality 
	\begin{equation}\nonumber
		\inf\left\{\capE_{s,g}(E) \mid |E|=m\right\} \geq  \inf\left\{\widetilde{\capE}_{s,g}(\{E^k\}_k) \mid \sum_{k=1}^{\infty} |E^k| = m \right\}
	\end{equation}
	holds true. Hence, it remains for us to prove the other inequality. To see this, we take any minimizing sequence $\{\{E^k_n\}_{k}\}_{n}$ for the generalized functional $\widetilde{\capE}_{s,g}$. Then it follows that, for any $\varepsilon>0$, there exists a number $n_0 \in \mathN$  such that 
	\begin{equation}\label{minimizingInequality}
		\widetilde{\capE}_{s,g}(\{E^k_n\}_k) \leq \widetilde{E}_{s,g}[m] + \varepsilon
	\end{equation} 
	for any $n \geq n_0$. Since the minimum is attained with a minimizing sequence of which each element is bounded, we may assume that $E^k_n$ is bounded for each $k,\,n\in\mathN$. In the sequel, we fix one $n \in \mathN$ with $n \geq n_0$ until we give another remark.
	
	First of all, we want to show that each element $\{E^k_n\}_{k}$ of the minimizing sequence $\{\{E^k_n\}_k\}_{n}$ can be chosen as an element of finitely many sets. More precisely, we will show that there exist a number $K_n\in\mathN$ and a sequence $\{\widetilde{E}^k_n\}_{k=1}^{K_n}$ such that the energy of $\{\widetilde{E}^k_n\}_{k=1}^{K_n}$ is smaller than that of $\{E^k_n\}_{k}$ and it converges to $\widetilde{E}_{s,g}[m]$ as $n \to \infty$. We set $m^{k}_n \coloneqq |E^{k}_n|$ for any $k \in \mathN$ and, since $\sum_{k=1}^{\infty} m^{k}_n = m <\infty$, we can observe that $m^{k}_n \to 0$ as $k \to \infty$ and, moreover, $\mu_{\ell} \coloneqq \sum_{k=\ell+1}^{\infty} m^{k}_n \to 0$ as $\ell \to \infty$. Then, we can choose $\widetilde{k} \in \mathN$ such that $m^{\widetilde{k}}_n \geq \frac{m}{2^{\widetilde{k}+1}}$. Indeed, if not, it follows that $m^k_n < \frac{m}{2^{k+1}}$ for any $k\in\mathN$. Then, we have that $m = \sum_{k=1}^{\infty}m^k < \sum_{k=1}^{\infty} \frac{m}{2^{k+1}} = \frac{m}{2}$, which is a contradiction. Now using the sets $\{E^{k}_n\}_{k=1}^{\infty}$, we construct a new family of sets $\{\widetilde{E}^{k}_n\}_{k=1}^{K_n}$ for some $K_n \in \mathN$, depending on $n$, in the following manner; we first choose $K_n \in \mathN$ so large that $K_n \geq \widetilde{k}$ and set $\widetilde{E}^{k}_n \coloneqq E^{k}_n$ for any $k \in \{1,\cdots ,\,K_n\}$ with $k \neq \widetilde{k}$ and $\widetilde{E}^{\widetilde{k}}_n \coloneqq \lambda_n\,E^{\widetilde{k}}_n$ where $\lambda^{N}_n \coloneqq \frac{m^{\widetilde{k}}_n + \mu_{K_n}}{m^{\widetilde{k}}_n}$. Then, we have the volume identity that
	\begin{equation}\label{minimizingSeqFinitenessExisGeneMini01}
		\sum_{k=1}^{K_n} \left| \widetilde{E}^{k}_n \right| = \sum_{k=1, \,k\neq\widetilde{k}}^{K_n} \left| E^{k}_n \right| + \lambda^N_n\, |E^{\widetilde{k}}_n| = \sum_{k=1,\,k\neq\widetilde{k}}^{K_n} m^{k}_n + m^{\widetilde{k}}_n + \mu_{K_n} = m.
	\end{equation}
	Now we compute the energy for $\{\widetilde{E}^{k}_n\}_{k=1}^{K_n}$ in order to show that the total energy of each elements of $\{\widetilde{E}^{k}_n\}_{k=1}^{K_n}$ is more efficient than that of the original sequence $\{E^{k}_n\}_{k=1}^{\infty}$. From the definition of $\lambda_n \geq 1$ and Lemma \ref{lemmaScalingEnergy}, we have that
	\begin{align}\label{minimizingSeqFinitenessExisGeneMini02}
		\sum_{k=1}^{K_n} \capE_{s,g}(\widetilde{E}^{k}_n) &\leq \sum_{k=1,\,k\neq\widetilde{k}}^{K_n} \capE_{s,g}(E^{k}_n) + \lambda^{2N}_n\,\capE_{s,g}(E^{\widetilde{k}}_n) \nonumber\\
		&= \sum_{k=1}^{\infty} \capE_{s,g}(E^{k}_n) + \left(\lambda^{2N}_n - 1\right)\,\capE_{s,g}(E^{\widetilde{k}}_n) - \sum_{k=K_n+1}^{\infty}\capE_{s,g}(E^{k}_n) \nonumber\\
		&\leq \sum_{k=1}^{\infty} \capE_{s,g}(E^{k}_n) + \frac{2^{\widetilde{k}+1}\,E_{s,g}[m]}{m}\,\mu_{K_n} - \sum_{k=K_n+1}^{\infty} P_s(E^k_n).
	\end{align}
	Here, in the last inequality, we have also used \eqref{minimizingSeqFinitenessExisGeneMini01}. From the isoperimetric inequality of $P_s$, \eqref{minimizingSeqFinitenessExisGeneMini02}, and \eqref{minimizingInequality}, we further obtain that
	\begin{align}
		\sum_{k=1}^{K_n} \capE_{s,g}(\widetilde{E}^{k}_n) &\leq \sum_{k=1}^{\infty} \capE_{s,g}(E^k_n) + \frac{2^{\widetilde{k}+1}\,E_{s,g}[m]}{m}\, \mu_{K_n} - C\sum_{k=K_n+1}^{\infty} \left(m^{k}_n\right)^{\frac{N-s}{N}}  \nonumber\\
		&\leq \sum_{k=1}^{\infty} \capE_{s,g}(E^k_n) + \frac{2^{\widetilde{k}+1}\,E_{s,g}[m]}{m}\, \mu_{K_n} - C \left(\sum_{k=K_n+1}^{\infty} m^{k}_n\right)^{\frac{N-s}{N}} \nonumber\\
		&= \sum_{k=1}^{\infty} \capE_{s,g}(E^k_n) + \frac{2^{\widetilde{k}+1}\,E_{s,g}[m]}{m}\, \mu_{K_n} - C \left( \mu_{K_n} \right)^{\frac{N-s}{N}}. \nonumber
	\end{align}
	Recalling the vanishing property of $\mu_{\ell}$ as $\ell \to \infty$ and taking the number $K_n$ so large that $K_n \geq \widetilde{k}$ and 
	\begin{equation}\nonumber
		\frac{2^{\widetilde{k}+1}\,E_{s,g}[m]}{m}\, \mu_{K_n} - C \left( \mu_{K_n} \right)^{\frac{N-s}{N}} < 0,
	\end{equation}
	then we finally obtain, from \eqref{minimizingInequality} and \eqref{minimizingSeqFinitenessExisGeneMini01}, that 
	\begin{equation}\nonumber
		\widetilde{E}_{s,g}[m] \leq \sum_{k=1}^{K_n} \capE_{s,g}(\widetilde{E}^{k}_n) < \sum_{k=1}^{\infty} \capE_{s,g}(E^k_n) \leq \widetilde{E}_{s,g}[m] + \varepsilon
	\end{equation}
	for large $n\in\mathN$. Finally letting $n \to \infty$, we obtain that the sum $\sum_{k=1}^{K_n} \capE_{s,g}(\widetilde{E}^{k}_n)$ converges to $\widetilde{E}_{s,g}[m]$ as $n \to \infty$. This completes the proof of the claim.
	
	Now we are prepared to prove the other inequality of the claim by using assumption $(\mathrm{g}4)$ saying that $g$ vanishes at infinity. From the fact that $\sum_{k=1}^{K_n}|E^k_n| = m$, we can choose one $k'\in\mathN$ with $|E^{k'}_n| > 0$. Since we have assumed that the sets $\{E^k_n\}_{k=1}^{K_n}$ are bounded, we can choose the points $\{z^{k}_n\}_{k=1,\,k \neq k'}^{K_n}$ such that each set $E^k_n + R\,z^{k}_n$ is far away from the others for sufficiently large $R>1$. We can thus compute the energy as follows; from the translation invariance of $\capE_{s,g}$, it holds that
	\begin{align}
		\sum_{k=1}^{K_n} \capE_{s,g}(E^k_n) &= \sum_{k=1,\,k\neq k', \ell}^{K_n} \capE_{s,g}(E^k_n) + \capE_{s,g}(E^{k'}_n) + \capE_{s,g}(E^{\ell}_n) \nonumber\\
		&= \sum_{k=1,\,k\neq k', \ell}^{K_n} \capE_{s,g}(E^k_n) + \capE_{s,g}(E^{k'}_n) + \capE_{s,g}(E^{\ell}_n + R\,z^{\ell}_n) \nonumber\\
		&= \sum_{k=1,\,k\neq k', \ell}^{K_n} \capE_{s,g}(E^k_n) + \capE_{s,g}(E^{k'}_n \cup (E^{\ell}_n + R\,z^{\ell}_n)) \nonumber\\
		&\qquad + 2\int_{E^{k'}_n} \int_{E^{\ell}_n + R\,z^{\ell}_n} \frac{dx\,dy}{|x-y|^{N+s}} - 2\int_{E^{k'}_n} \int_{E^{\ell}_n + R\,z^{\ell}_n}g(x-y)\,dx\,dy \nonumber
	\end{align}
	for any $\ell \in \{1,\cdots,\,K_n\}$ with $\ell \neq k'$ and sufficiently large $R>1$. Thus, we obtain that
	\begin{align}\nonumber
		\sum_{k=1,\,k\neq k', \ell}^{K_n} \capE_{s,g}(E^k_n) + \capE_{s,g}(E^{k'}_n \cup (E^{\ell}_n + R\,z^{\ell}_n)) &\leq \sum_{k=1}^{K_n} \capE_{s,g}(E^k_n) \nonumber\\
		&\qquad + 2\int_{E^{k'}_n} \int_{E^{\ell}_n + R\,z^{\ell}_n}g(x-y)\,dx\,dy
	\end{align} 
	for any $\ell \in \{1,\cdots,\,K_n\}$ with $\ell \neq k'$ and sufficiently large $R>1$. By repeating the same argument finite times for the rest of the sets $\{E^k_n\}_{k=1,\,k \neq k',\ell}^{K_n}$ with sufficiently large $R>1$ and from the translation invariance of $\capE_{s,g}$, we can derive the inequality
	\begin{align}\label{keyEstimateReductionOneElementGeneMini02}
		\capE_{s,g}\left( E^{k'}_n \cup \bigcup_{k=1,\,k\neq k'}^{K_n}\left( E^k_n + R\,z^{k}_n \right) \right) &\leq \sum_{k=1}^{K_n} \capE_{s,g}(E^k_n) \nonumber\\
		&\qquad + 2\sum_{k=1}^{K_n-1}\sum_{\ell=k+1}^{K_n} \int_{F^k_n(R)} \int_{F^{\ell}_n(R)} g(x-y)\,dx\,dy
	\end{align}
	where we define the sets $\{F^k_n(R)\}_{k=1}^{K_n}$ in such a way that $F^k_n(R) \coloneqq E^{k}_n + R\,z^{k}_n$ if $k \neq k'$ and $F^{k'}_n(R) \coloneqq E^{k'}_n$. Note that the sets $\{F^k_n(R)\}_{k=1}^{K_n}$ satisfy
	\begin{equation}\label{mutuallyFarAway}
		\dist(F^k_n(R), \,F^{\ell}_n(R)) \xrightarrow[R \to \infty]{} \infty
	\end{equation}
	for any $k, \, \ell \in \{1,\cdots,\,K_n\}$ with $k \neq \ell$. Since $\sum_{k=1}^{K_n}|E^k_n| = m$ and $E^{k'}_n \cup \bigcup_{k=1,\,k\neq k'}^{K_n}\left( E^k_n + R\,z^{i_p} \right)$ are the union of disjoint sets, we have that 
	\begin{equation}\nonumber
		\left| E^{k'}_n \cup \bigcup_{k=1,\,k\neq k'}^{K_n}\left( E^k_n + R\,z^{i_p} \right)\right| = \sum_{k=1}^{K_n} |E^{k}_n| = m.
	\end{equation}
	Thus, from \eqref{minimizingInequality} and \eqref{keyEstimateReductionOneElementGeneMini02}, we obtain
	\begin{align}\label{keyEstimateReductionOneElementGeneMini03}
		E_{s,g}[m] &\leq \capE_{s,g}\left( E^{k'}_n \cup \bigcup_{k=1,\,k\neq k'}^{K_n}\left( E^k_n + R\,z^{k}_n \right) \right) \nonumber\\
		&\leq \sum_{k=1}^{K_n} \capE_{s,g}(E^k_n) \nonumber\\
		&\qquad  + 2\sum_{k=1}^{K_n}\sum_{\ell=k+1}^{K_n} \int_{F^k_n(R)} \int_{F^{\ell}_n(R)} g(x-y)\,dx\,dy \nonumber\\
		&\leq \widetilde{E}_{s,g}[m] + \varepsilon \nonumber\\
		&\qquad + 2\sum_{k=1}^{K_n}\sum_{\ell=k+1}^{K_n} \int_{F^k_n(R)} \int_{F^{\ell}_n(R)} g(x-y)\,dx\,dy
	\end{align}
	Hence, if we show that the last term of the right-hand side in \eqref{keyEstimateReductionOneElementGeneMini03} converges to zero as $R \to \infty$ for each $n \geq n_0$, then we conclude that the inequality 
	\begin{equation}\nonumber
		\inf\left\{\capE_{s,g}(E) \mid |E|=m \right\} = E_{s,g}[m] \leq \widetilde{E}_{s,g}[m] =  \inf\left\{\widetilde{\capE}_{s,g}(\{E^k\}_k) \mid \sum_{k=1}^{\infty} |E^k| = m \right\}
	\end{equation}
	holds and this completes the proof of the lemma. To conclude the proof of the lemma, it is sufficient to show that, under assumption $(\mathrm{g}4)$, the convergence
	\begin{equation}\nonumber
		\sum_{k=1}^{K_n}\sum_{\ell=k+1}^{K_n} \int_{F^k_n(R)} \int_{F^{\ell}_n(R)} g(x-y)\,dx\,dy \xrightarrow[R \to \infty]{} 0 
	\end{equation}
	holds for each $n \geq n_0$. We fix $n \geq n_0$. From assumption $(\mathrm{g}4)$, we have that, for any $\varepsilon>0$, there exists a constant $R(\varepsilon)>0$ such that $g(z) < \varepsilon$ for any $|z| \geq R(\varepsilon)$. On the other hand, from \eqref{mutuallyFarAway}, we can also choose a constant $R'(\varepsilon)>0$ such that $|x-y| \geq R(\varepsilon)$ for any $R>R'(\varepsilon)$, $x \in F^k_n(R)$, $y \in F^{\ell}_n(R)$, and $k,\,\ell \in \{1,\cdots,\,K_n\}$ with $k \neq \ell$. Thus, taking these into account, we obtain that, for any $R > R'(\varepsilon)$,
	\begin{equation}\nonumber
		\sum_{k=1}^{K_n}\sum_{\ell=k+1}^{K_n} \int_{F^k_n(R)} \int_{F^{\ell}_n(R)} g(x-y)\,dx\,dy < \varepsilon \,\sum_{k=1}^{K_n}|F^k_n(R)| \sum_{\ell=1}^{K_n} |F^\ell_n(R)|.
	\end{equation}  
	Recalling the definition of the sets $\{F^k_n(R)\}_{k}$, we have that $\sum_{k=1}^{K_n}|F^k_n(R)| \leq m$. Therefore, we obtain that
	\begin{equation}\nonumber
		\sum_{k=1}^{K_n}\sum_{\ell=k+1}^{K_n} \int_{F^k_n(R)} \int_{F^{\ell}_n(R)} g(x-y)\,dx\,dy < m^2\,\varepsilon
	\end{equation}
	for any $R > R'(\varepsilon)$ and this completes the proof of the claim.

\end{proof}

Now we prove Theorem \ref{theoremExistGeneralizedMiniAnyVolume}, namely, the existence of generalized minimizers of $\widetilde{\capE}_{s,g}$ under the assumptions $(\mathrm{g}1)$, $(\mathrm{g}2)$, and $(\mathrm{g}4)$ in Section \ref{sectionPreliminary}. 
\begin{proof}[Proof of Theorem \ref{theoremExistGeneralizedMiniAnyVolume}]
	Let $m>0$. Thanks to Lemma \ref{lemmaModifyMinimizingSeq}, it is sufficient to take any sequence $\{E_n\}_{n\in\mathN}$ such that $|E_n| = m$ for any $n\in\mathN$ and
	\begin{equation}\label{minimizingSequenceGeneMini}
		\lim_{n \to \infty} \capE_{s,g}(E_n) = \widetilde{E}_{s,g}[m]
	\end{equation}
	instead of taking a minimizing sequence for $\widetilde{E}_{s,g}[m]$. Thus, we can now apply the same argument as in the proof of the existence of minimizers of $\capE_{s,g}$ in Theorem \ref{theoremExistMiniAnyVolumeFasterDecay} because we assume that $g$ satisfies the assumptions $(\mathrm{g}1)$ and $(\mathrm{g}2)$. This enables us to choose a sequence of a finite number of measurable sets $\{G^i\}_{i=1}^H$ with $H \in \mathN$ such that
	\begin{equation}\label{limitMinimizerManyComponents}
		\sum_{i=1}^{H}\capE_{s,g}(G^i) \leq \liminf_{n \to \infty}\capE_{s,g}(E_n), \quad \sum_{i=1}^{H} |G^i| = m.
	\end{equation} 
	Moreover, each set $G^i$ is a minimizer of $\capE_{s,g}$ among sets of volume $|G^i|$. Therefore, from \eqref{minimizingSequenceGeneMini} and \eqref{limitMinimizerManyComponents}, we conclude that the sequence $\{G^i\}_{i=1}^H$ is a generalized minimizer of $\widetilde{\capE}_{s,g}$ with $\sum_{i=1}^{H}|G^i| = m$ as we desired.
\end{proof}

\section{Asymptotic behavior of minimizers for large volumes}\label{sectionAsymptoticMiniLargeVol}
In this section, we study the asymptotic behavior of minimizers of $\capE_{s,g}$ with large volumes under the assumption that the kernel $g$ decays sufficiently fast. To see this, we first prove the $\Gamma$-convergence in $L^1$-topology of the functional associated with Problem \eqref{minimizationScalingModifiedProbelm} to the fractional perimeter $P_s$ as $m$ goes to infinity. Since it is well-known that a sequence of minimizers for a functional converges to a minimizer of its $\Gamma$-limit, we can derive the convergence of a sequence of the minimizers to the unit ball, by rescaling, up to translations.

\subsection{$\Gamma$-convergence of $\widehat{\capE}^{\lambda}_{s,g}$ to the fractional perimeter as $\lambda \to \infty$}
Before proving Theorem \ref{theoremAsympMiniLargeVolume}, we establish the $\Gamma$-convergence result for the energy $\capE^{\lambda}_{s,g}$ under the assumption that the kernel $g$ decays sufficiently fast. Before stating the claim, we give several notations and the definition of the functional $\capF^{\lambda}_{s,g}$ on $L^1(\mathR^N)$. First, we recall the definition of the $s$-fractional Sobolev semi-norm $[f]_{W^{s,1}}(\mathR^N)$ as follows:
\begin{equation}\nonumber
	[f]_{W^{s,1}}(\mathR^N) = \frac{1}{2}\int_{\mathR^N} \int_{\mathR^N}\frac{|f(x) - f(y)|}{|x-y|^{N+s}}\,dx\,dy
\end{equation}
for $f \in L^1$. Note that $[\chi_E]_{W^{s,1}}(\mathR^N) = P_s(E)$ for any measurable set $E \subset \mathR^N$. As is shown in \cite[Proposition 4.2 and Corollary 4.4]{BLP}, it follows that any integrable function of bounded variation is also finite with respect to the fractional semi-norm $[\cdot]_{W^{s,1}}$. Secondly, in order to study the $\Gamma$-convergence of the sequence  $\{\widehat{\capE}^{\lambda}_{s,g}\}_{\lambda>1}$ given in Proposition \ref{propositionRescaledProblem}, we define the functional $\widehat{\capF}^{\lambda_n}_{s,g}$ as
\begin{align}\label{functionalGammaConv}
	\widehat{\capF}^{\lambda_n}_{s,g}(f) \coloneqq 
	\left\{
	\begin{array}{ll}
		 [f]_{W^{s,1}}(\mathR^N) - & \displaystyle  \frac{1}{2}\int_{\mathR^N}\int_{\mathR^N} |f(x) - f(y)|\, g_{\lambda}(x-y) \,dx\,dy  \\
		 \\
		 & \text{if $f=\chi_{F}$ for some bounded set $F \subset \mathR^N$ with $P_s(F) < \infty$,} \\
		 \\
		 +\infty \quad  &\text{otherwise.}
	\end{array}
	\right.
\end{align}
Note that the functional $\widehat{\capF}^{\lambda}_{s,g}(f)$ for any $\lambda>0$ is well-defined. Moreover, if $f =\chi_E$ for some bounded set $E$ with $P_s(E) < \infty$, then we can easily see that $\widehat{\capF}^{\lambda}_{s,g}(f) = \widehat{\capE}^{\lambda}_{s,g}(E)$.

Now we prove the $\Gamma$-convergence of the functional $\widehat{\capF}^{\lambda_n}_{s,g}$ to $\widehat{\capF}^{\infty}_{s}$ (we give the definition of $\widehat{\capF}^{\infty}_{s}$ in the following proposition) as $n \to \infty$ in the $L^1$-topology.
\begin{proposition}[$\Gamma$-convergence to the $s$-fractional semi-norm] \label{propositionGammaConvergenceNonlocalEnergy}
	Let $\{\lambda_n\}_{n\in\mathN}$ be any sequence of positive real numbers going to infinity as $n\to\infty$. Assume that the kernel $g : \mathR^N \setminus \{0\} \to \mathR$ is radially symmetric and satisfies the assumptions $(\mathrm{g}1)$, $(\mathrm{g}2)$, and $(\mathrm{g}5)$ in Section \ref{sectionPreliminary}. Then the sequence $\{\widehat{\capF}^{\lambda_n}_{s,g}\}_{n\in\mathN}$ $\Gamma$-converges, with respect to $L^1$-topology, to the functional $\widehat{\capF}^{\infty}_{s}$ defined by
	\begin{align}
		\widehat{\capF}^{\infty}_{s}(f) \coloneqq 
		\left\{
		\begin{array}{ll}
			\displaystyle [f]_{W^{s,1}}(\mathR^N) \quad &\text{if $f=\chi_{F}$ for some bounded $F \subset \mathR^N$ with $P_s(F) < \infty$,} \nonumber\\
			\\
			+\infty \quad &\text{otherwise.} \nonumber
		\end{array}
		\right.
	\end{align}
\end{proposition}
\begin{remark}
	Recall that, in this paper, we assume that the kernel $g$ is locally integrable in $\mathR^N$, especially near the origin; however, Proposition \ref{propositionGammaConvergenceNonlocalEnergy} is still valid even if $g$ is not integrable in the ball centred at the origin. This is because the assumption that $g(x) \leq |x|^{-(N+s)}$ for $x \neq 0$ is sufficient enough for the functional \eqref{functionalGammaConv} to be well-defined.
\end{remark}
\begin{proof}
	We recall the definition of the $\Gamma$-convergence. We say that $\{\widehat{\capF}^{\lambda_n}_{s,g}\}_{n\in\mathN}$ $\Gamma$-converges to $\widehat{\capF}^{\infty}_{s}$ with respect to $L^1$-topology if the two estimates hold
	\begin{equation}\nonumber
		\Gamma_{L^1-}\limsup_{n\to\infty}\widehat{\capF}^{\lambda_n}_{s,g}(f) \leq \widehat{\capF}^{\infty}_{s}(f), \quad \widehat{\capF}^{\infty}_{s}(f) \leq \Gamma_{L^1-}\liminf_{n\to\infty}\widehat{\capF}^{\lambda_n}_{s,g}(f)
	\end{equation}
	for any $f \in L^1(\mathR^N)$, where we set
	\begin{equation}\label{defgammaConveSup}
		\Gamma_{L^1-}\limsup_{n\to\infty}\widehat{\capF}^{\lambda_n}_{s,g}(f) \coloneqq \min\left\{\limsup_{n\to\infty} \widehat{\capF}^{\lambda_n}_{s,g}(f_n) \mid \text{$f_n \xrightarrow[n\to\infty]{} f$ in $L^1(\mathR^N)$} \right\}
	\end{equation}
	and
	\begin{equation}\label{defgammaConveInf}
		\Gamma_{L^1-}\liminf_{n\to\infty}\widehat{\capF}^{\lambda_n}_{s,g}(f) \coloneqq \min\left\{\liminf_{n\to\infty} \widehat{\capF}^{\lambda_n}_{s,g}(f_n) \mid \text{$f_n \xrightarrow[n\to\infty]{} f$ in $L^1(\mathR^N)$} \right\}.
	\end{equation}
	Note that the minimum in \eqref{defgammaConveSup} and \eqref{defgammaConveInf} is attained by the diagonal argument.
	
	First of all, we prove that $\Gamma_{L^1-}\limsup_{n\to\infty}\widehat{\capF}^{\lambda_n}_{s,g}(f) \leq \widehat{\capF}^{\infty}_{s}(f)$ for any $f \in L^1(\mathR^N)$. In the case that $f$ is not a characteristic function of some bounded set with a finite nonlocal perimeter, we obviously have that $\widehat{\capF}^{\infty}_{s}(f) = \infty$ and the inequality holds. Thus, we may assume that $f = \chi_F$ for a bounded set $F \subset \mathR^N$ with $P_s(F) < \infty$. Setting a sequence $\{f_n\}_{n\in\mathN}$ as $f_n = f = \chi_F$ for any $n\in\mathN$, we obtain, from the non-negativity of $g$, that
	\begin{equation}\nonumber
		\widehat{\capF}^{\lambda_n}_{s,g}(f_n) \leq \widehat{\capF}^{\infty}_{s}(f)
	\end{equation}
	for any $n\in\mathN$ and thus, it follows that $\Gamma_{L^1-}\limsup_{n\to\infty}\widehat{\capF}^{\lambda_n}_{s,g}(f) \leq \widehat{\capF}^{\infty}_{s}(f)$.
	
	Next we prove that $\widehat{\capF}^{\infty}_{s}(f) \leq \Gamma_{L^1-}\liminf_{n\to\infty}\widehat{\capF}^{\lambda_n}_{s,g}(f)$ for any $f \in L^1(\mathR^N)$. We take any sequence $\{f_n\}_{n\in\mathN} \subset L^1(\mathR^N)$ such that $f_n \to f$ in $L^1$ as $n \to \infty$. In the case that $f$ is not a characteristic function of some bounded set with a finite nonlocal perimeter, we claim that there exists a number $n_0 \in \mathN$ such that $f_n$ is also not a characteristic function of a measurable set for any $n \geq n_0$. Indeed, we suppose by contradiction that there exists a subsequence $\{f_{n_k}\}_{k\in\mathN}$ such that $f_{n_k} = \chi_{F_{n_k}}$ for some measurable set $F_{n_k} \subset \mathR^N$ for any $k \in \mathN$. Since $f_{n_k} \to f$ in $L^1$ as $k \to \infty$ and $f_{n_k} \in \{0,\,1\}$ for any $k\in\mathN$, we obtain that $f \in \{0,\,1\}$ a.e. in $\mathR^N$ and $f$ can be written as $f = \chi_F$ for some measurable $F\ \subset \mathR^N$. This contradicts the assumption that $f$ is not a characteristic function. Hence, we conclude that, for large $n \in \mathN$, $\widehat{\capF}^{\lambda_n}_{s,g}(f_n) = \infty$ and the claim holds true. Thus, in the sequel, we may assume that $f = \chi_F$ for some bounded set $F \subset \mathR^N$ with $P_s(F) < \infty$.
	
	Under the above assumption, we first compute the second term of the functional $\widehat{\capF}^{\lambda_n}_{s,g}$ in \eqref{functionalGammaConv}. Let $\varepsilon \in (0,\,1)$. From assumption $(\mathrm{g}5)$, we can choose a constant $R_\varepsilon > 1$ such that $g(x) \leq \frac{\varepsilon}{|x|^{N+s}}$ for $|x| \geq R_\varepsilon$. Then, from the definition of $g_{\lambda_n}$ for any $n\in\mathN$, we have that
	\begin{align}\label{splittingSecondTermEnergy}
		&\int_{\mathR^N}\int_{\mathR^N} |f(x) - f(y)| \, g_{\lambda_n}(x-y) \,dx\,dy \nonumber\\
		&= \iint_{\{(x,\,y) \mid \, \lambda_n|x-y| < R_\varepsilon \}} |f(x) - f(y)| \, g_{\lambda_n}(x-y) \,dx\,dy \nonumber\\
		&\qquad + \iint_{\{(x,\,y) \mid \, \lambda_n|x-y| \geq R_\varepsilon \}} |f(x) - f(y)| \, g_{\lambda_n}(x-y) \,dx\,dy \nonumber\\ 
		&\leq  \iint_{\{(x,\,y) \mid \, \lambda_n|x-y| < R_\varepsilon \}} \frac{|f(x) - f(y)|}{|x-y|^{N+s}} \,dx\,dy \nonumber\\ 
		&\qquad + \varepsilon \iint_{\{(x,\,y) \mid \, \lambda_n|x-y| \geq R_\varepsilon \}} \frac{|f(x) - f(y)|}{|x-y|^{N+s}} \,dx\,dy
	\end{align}
	for any $n \in \mathN$. Thus, from the definition of $\widehat{\capE}^{\lambda_n}_{s,g}$ and \eqref{splittingSecondTermEnergy}, we can obtain
	\begin{align}\label{splittingSecondTermEnergy02}
		\widehat{\capF}^{\lambda_n}_{s,g}(f_n) &\geq  [f_n]_{W^{s,1}}(\mathR^N) -  \frac{1}{2}\iint_{\{(x,\,y) \mid \, \lambda_n|x-y| < R_\varepsilon \}} \frac{|f_n(x) - f_n(y)|}{|x-y|^{N+s}} \,dx\,dy \nonumber\\ 
		&\qquad - \frac{\varepsilon}{2}\iint_{\{(x,\,y) \mid \, \lambda_n|x-y| \geq R_\varepsilon \}} \frac{|f_n(x) - f_n(y)|}{|x-y|^{N+s}} \,dx\,dy \nonumber\\
		&\geq \frac{1-\varepsilon}{2}\iint_{\{(x,\,y) \mid \, \lambda_n|x-y| \geq R_\varepsilon \}} \frac{|f_n(x) - f_n(y)|}{|x-y|^{N+s}} \,dx\,dy 
	\end{align} 
	for any $\varepsilon \in (0,\,1)$ and $n\in\mathN$. Thus, letting first $n \to \infty$ and then $\varepsilon \to 0$ with Fatou's lemma and the monotone convergence theorem, we finally obtain
	\begin{align}
		\liminf_{n \to \infty}\widehat{\capF}^{\lambda_n}_{s,g}(f_n) &\geq \limsup_{\varepsilon \to 0}\frac{1-\varepsilon}{2}\iint_{\mathR^N \times \mathR^N}  \frac{|f(x) - f(y)|}{|x-y|^{N+s}} \,dx\,dy \nonumber\\
		&= \frac{1}{2}\int_{\mathR^N}\int_{\mathR^N}\frac{|f(x) - f(y)|}{|x-y|^{N+s}} \,dx\,dy = [f]_{W^{s,1}}(\mathR^N). \nonumber
	\end{align} 
	Therefore, from the above arguments, we complete the proof. 
\end{proof}

\subsection{Convergence of minimizers of $\widehat{\capE}^{\lambda}_{s,g}$ to the ball as $\lambda \to \infty$}
Now we are prepared to prove Theorem \ref{theoremAsympMiniLargeVolume}. In this theorem, we impose on $g$ the assumptions that $g$ is radially symmetric and decays sufficiently fast. One important difference between the assumptions on $g$ of Theorem \ref{theoremExistMiniAnyVolumeFasterDecay} and \ref{theoremAsympMiniLargeVolume} is that the decay of the kernel $g$ in Theorem \ref{theoremAsympMiniLargeVolume} is strictly faster than that in Theorem \ref{theoremExistMiniAnyVolumeFasterDecay}. To observe the convergence, we consider Problem \ref{minimizationScalingModifiedProbelm} and finally we take the limit $\lambda \to \infty$ instead of Problem \ref{minimizationGeneralizedFunctional} with the limit $m \to \infty$ for convenience.

The strategy for the proof of the asymptotic behavior of the minimizers is as follows; in contrast to the idea for studying the asymptotic behavior of minimizers, for instance, in \cite{FFMMM, Pegon, Carazzato}, we may not be able to employ directly the quantitative isoperimetric inequality for the nonlocal perimeter $P_s$ to use a Fuglede-type argument. This is because we may have the volume of the symmetric difference between minimizers and the balls naively bounded by the volume and perimeter of the symmetric difference. Since we are dealing with minimizers with large volumes, it is not obvious that the bound of the symmetric difference could give us the $L^1$-convergence of minimizers to the ball. Therefore, we adopt another strategy in the following way; we first take any sequence $\{F_n\}_{n}$ of the minimizers for $\widehat{\capE}^{\lambda_n}_{s,g}$ with $|F_n|=|B_1|$. Then we apply so-called ``concentration-compactness" lemma that we use to show the existence of minimizers in Section \ref{sectionExisMiniFastDecayGAnyVol} and \ref{sectionExistGeneMiniAnyVol}. As a consequence of the lemma, we can choose a sequence of sets $\{G^i\}_{i}$ and points $\{z_n^i\}_{i,n}$ such that, up to extracting a subsequence, 
\begin{equation}\label{keyInfoAsymptoticLargeMini}
	\sum_{i}P_s(G^i) \leq \liminf_{n \to \infty}\widehat{\capE}^{\lambda_n}_{s,g}(F_n), \quad F_n - z_n^i \xrightarrow[n \to \infty]{} G^i \quad \text{in $L^1_{loc}$}, \quad \sum_{i}|G^i| = |B_1|
\end{equation}
thanks to the assumptions on $g$. Then, from the isoperimetric inequality of $P_s$ and the minimality of $F_n$, we can actually obtain that each $G^i$ coincides with the Euclidean ball, up to translations and negligible sets, whenever $|G^i|>0$. Finally, from \eqref{keyInfoAsymptoticLargeMini}, we can show that the only one element in $\{G^i\}_{i}$ has a positive volume and its volume is equal to $|B_1|$. From Brezis-Lieb lemma, the convergence in \eqref{keyInfoAsymptoticLargeMini} is improved to the $L^1$-convergence. Combining the $\Gamma$-convergence result, we conclude the proof.  

\begin{proof}[Proof of Theorem \ref{theoremAsympMiniLargeVolume}]
	Let $\{\lambda_n\}_{n\in\mathN}$ be any sequence going to infinity as $n \in \mathN$ and we take any sequence $\{F_n\}_{n\in\mathN}$ of the minimizers for $\widehat{\capE}^{\lambda_n}_{s,g}$ with $|F_n| = |B_1|$ for any $n \in \mathN$. From the assumption $(\mathrm{g}5)$, we can choose a constant $\gamma \in (0,\,1)$ such that $g_{\lambda_n}(x) \leq \gamma|x|^{-(N+s)}$ for any $|x| \neq 0$. From the minimality of $F_n$ for each $n\in\mathN$, we have that 
	\begin{equation}\nonumber
		P_s(F_n) \leq P_s(B_1) + P_{g_{\lambda_n}}(F_n) = P_s(B_1) + \gamma \, P_s(F_n)
	\end{equation}
	for any $n \in \mathN$ and thus, we obtain that $\{P_s(F_n)\}_{n}$ is uniformly bounded with respect to $n$, namely, $\sup_{n\in\mathN}P_s(F_n) \leq (1-\gamma)^{-1}P_s(B_1) < \infty$. As a consequence of the uniform bound of $\{P_s(F_n)\}_{n}$, we can now apply the same method as in the proof of Theorem \ref{theoremExistMiniAnyVolumeFasterDecay} (see also \cite{dCNRV}) to the sequence $\{F_n\}_{n}$. Although we discuss in the proof of Theorem \ref{theoremExistMiniAnyVolumeFasterDecay}, we rewrite the argument in the sequel for convenience. 
	
	First of all, we decompose $\mathR^N$ into the unit cubes and denote by $\{Q_n^i\}_{i=1}^{\infty}$. We set $x_n^i \coloneqq |F_n \cap Q_n^i|$ and have that
	\begin{equation}\label{keyTechnicalAsymptotic01}
		\sum_{i=1}^{\infty} x_n^i = |F_n| = |B_1|
	\end{equation} 
	for any $n\in\mathN$. Moreover, from the isoperimetric inequality shown in \cite[Lemma 2.5]{dCNRV}, we obtain
	\begin{equation}\label{keyTechnicalAsymptotic02}
		\sum_{i=1}^{\infty} (x_n^i)^{\frac{N-s}{N}} \leq C\sum_{i=1}^{\infty}P_s(F_n; Q_n^i) \leq 2CP_s(F_n) \leq C_1 < \infty
	\end{equation}
	for any $n\in\mathN$, where $C$ and $C_1$ are the positive constants independent of $n$. Up to reordering the cubes $\{Q_n^i\}_{i}$, we may assume that $\{x_n^i\}_{i}$ is a non-increasing sequence for any $n\in\mathN$. Thus, applying the technical result shown in \cite[Lemma 4.2]{GoNo} or \cite[Lemma 7.4]{dCNRV} with \eqref{keyTechnicalAsymptotic01} and \eqref{keyTechnicalAsymptotic02}, we obtain that
	\begin{equation}\label{keyTechnicalAsymptotic03}
		\sum_{i=k+1}^{\infty} x_n^i \leq \frac{C_2}{k^{\frac{s}{N}}} 
	\end{equation}
	for any $k\in\mathN$ where $C_2$ is the positive constant independent of $n$ and $k$. Hence, by using the diagonal argument, we have that, up to extracting a subsequence, $x_n^i \to \alpha^i \in [0,\,|B_1|]$ as $n \to \infty$ for every $i\in\mathN$ and, from \eqref{keyTechnicalAsymptotic01} and \eqref{keyTechnicalAsymptotic03}, 
	\begin{equation}\label{identityLimitMeasuAsymptotic}
		\sum_{i=1}^{\infty} \alpha^i = |B_1|.
	\end{equation}
	Now we fix the centre of the cube $z_n^i \in Q_n^i$ for each $i$ and $n$. Up to extracting a further subsequence, we may assume that $|z_n^i - z_j^i| \to c^{ij} \in [0,\,\infty]$ as $n \to \infty$ for each $i,\,j \in \mathN$. As already seen in the above, we have the uniform bound of the sequence $\{P_s(F_n-z_n^i)\}_{n\in\mathN}$ and its upper-bound is independent of $i$ and thus, from the compactness, there exists a measurable set $G^i \subset \mathR^N$ such that, up to extracting a further subsequence, 
	\begin{equation}\nonumber
		\chi_{F_n-z_n^i} \xrightarrow[n \to \infty]{} \chi_{G^i} \quad \text{in $L^1_{loc}$-topology}.
	\end{equation}
	We define the relation $i \sim j$ for every $i,\,j\in\mathN$ as $c^{ij} < \infty$ and we denote by $[i]$ the equivalent class of $i$. Moreover, we define the set of the equivalent class by $\capI$. Then, by applying the same argument as in the proof of Theorem \ref{theoremExistMiniAnyVolumeFasterDecay}, it is easy to show that 
	\begin{equation}\nonumber
		\sum_{[i] \in \capI} |G^i| = |B_1|.
	\end{equation}

	As a first step, we want to show a sort of lower semi-continuity of the energy, more precisely, we will prove the following inequality;
	\begin{equation}\label{lowerSemicontiConcentrationAsymptotic}
		\sum_{[i] \in \capI}P_s(G^i) \leq \liminf_{n \to \infty}\widehat{\capE}^{\lambda_n}_{s,g}(F_n) = \liminf_{n \to \infty}\left(P_s(F_n) - P_{g_{\lambda_n}}(F_n) \right).
	\end{equation}
	Indeed, we first fix $M \in \mathN$ and $R>0$ and we take the equivalent classes $i_1,\cdots,\,i_M$. Notice that, if $p \neq q$, then $|z_n^{i_p} - z_n^{i_q}| \to \infty$ as $n\to\infty$ and thus we have that $\{z_n^{i_p}+Q_R\}_{p}$ are disjoint sets for large $n$ and 
	\begin{equation}\label{vanishingEnergyGeneralizedMini}
		\int_{z_n^{i_p} + Q_R}\int_{z_n^{i_q} + Q_R} \frac{1}{|x-y|^{N+s}}\,dx\,dy \xrightarrow[n \to \infty]{} 0
	\end{equation} 
	where $Q_R$ is the cube of side $R$. Then, by using a similar argument shown in the proof of the $\Gamma$-liminf inequality in Proposition \ref{propositionGammaConvergenceNonlocalEnergy} with \eqref{vanishingEnergyGeneralizedMini}, we can conduct the following similar argument: let $\varepsilon \in (0,\,1)$ and, from $(\mathrm{g}5)$, we can choose a constant $R_\varepsilon>1$ such that $g(x) \leq \frac{\varepsilon}{|x|^{N+s}}$ for any $|x| \geq R_\varepsilon$. Then it holds that
	\begin{align}\label{liminfInequalityAsymptotic}
		&\liminf_{n \to \infty}\left(P_s(F_n) - P_{g_{\lambda_n}}(F_n) \right) \nonumber\\
		&\geq (1-\varepsilon) \liminf_{n \to \infty}\left(\int_{F_n \cap A^{M,R}_n}\int_{F^c_n}\frac{\chi_{\{|x-y| \geq r_n^{\varepsilon}\}}(x,\,y)}{|x-y|^{N+s}} \,dx\,dy \right) \nonumber\\
		&\qquad + (1-\varepsilon) \liminf_{n \to \infty}\left(\int_{F_n \setminus A^{M,R}_n}\int_{A^{M,R}_n \setminus F_n}\frac{\chi_{\{|x-y| \geq r_n^{\varepsilon}\}}(x,\,y)}{|x-y|^{N+s}} \,dx\,dy \right) \nonumber\\
		&\qquad \quad + (1-\varepsilon)\liminf_{n \to \infty}2\sum_{p \neq q} \int_{z^{i_p}_n + Q_R}\int_{z^{i_q}_n + Q_R}\frac{\chi_{\{|x-y| \geq r_n^{\varepsilon}\}}(x,\,y)}{|x-y|^{N+s}}\,dx\,dy
	\end{align}
	for any $\varepsilon \in (0,\,1)$ where we set $r_n^{\varepsilon} \coloneqq \lambda^{-1}_n\,\varepsilon^{-\frac{1}{1-s}}$ for each $n$ and $A^{M,R}_n \coloneqq \cup_{p=1}^{M}(z^{i_p}_n + Q_R)$. Now we recall the inequality of double integrals;
	\begin{align}\label{subadditivityNonlocalPeri}
		&P_s(E;A) + P_s(E;B) \leq P_s(E; A \sqcup B) + 2\int_{A}\int_{B}\frac{dx\,dy}{|x-y|^{N+s}}
	\end{align}
	for any measurable disjoint sets $A,\,B\subset \mathR^N$, where we define $P_s(E;A) \coloneqq \int_{E \cap A}\int_{E^c} + \int_{E \setminus A}\int_{A \setminus E}$ for measurable sets $E,\,A \subset \mathR^N$ (we omit the integrand for simplicity). Hence, from \eqref{liminfInequalityAsymptotic}, \eqref{subadditivityNonlocalPeri}, and the lower semi-continuity of $P_s$ in $L^1_{loc}$-topology with Fatou's lemma, we obtain
	\begin{align}
		&\liminf_{n \to \infty}\left(P_s(F_n) - P_{g_{\lambda_n}}(F_n) \right) \nonumber\\
		&\geq (1-\varepsilon)\liminf_{n \to \infty} \sum_{p=1}^{M} \left(\int_{F_n \cap (z^{i_p}_n + Q_R)} \int_{F_n^c} \frac{\chi_{\{|x-y|\geq r_n^{\varepsilon}\}}(x,\,y)}{|x-y|^{N+s}}\,dx\,dy \right. \nonumber\\
		&\qquad \left. + \int_{F_n \setminus (z^{i_p}_n + Q_R)} \int_{(z^{i_p}_n + Q_R) \setminus F_n} \frac{\chi_{\{|x-y|\geq r_n^{\varepsilon}\}}(x,\,y)}{|x-y|^{N+s}}\,dx\,dy \right) \nonumber\\
		&\geq (1-\varepsilon)\sum_{p=1}^{M} P_s(G^{i_p} ; Q_R) \nonumber
	\end{align}
	for any $\varepsilon \in (0,\,1)$. Letting $R \to \infty$, $M \to \infty$, and $\varepsilon \to 0$, we finally conclude that the inequality \eqref{lowerSemicontiConcentrationAsymptotic} holds true. Taking into account all of the above arguments, we obtain the existence of sets $\{G^i\}_{[i]\in\capI}$ satisfying the properties that
	\begin{equation}\label{keyPropertylowSemicontiVolumeEqualParticlesAsymptotic}
		\sum_{[i] \in \capI} P_s(G^i) \leq \liminf_{n \to \infty} \widehat{\capE}^{\lambda_n}_{s,g}(F_n), \quad \sum_{[i] \in \capI} |G^{i}| = |B_1|. 
	\end{equation}

	Secondly, we want to show that each $G^{i}$ actually coincides, up to translations and negligible sets, with the Euclidean ball with the volume $|G^{i}|$, whenever $|G^{i}|>0$. Indeed, we first set $B_i$ as the ball of radius $r_i \coloneqq |B_1|^{-1/N}\,|G^i|^{1/N}$ for each $[i] \in \capI$. Then, from \eqref{keyPropertylowSemicontiVolumeEqualParticlesAsymptotic} and the minimality of $F_n$, we have that
	\begin{align}\label{energyMinimalityEachPrticleAsymptotic}
		\sum_{[i] \in \capI} \left( P_s(G^{i}) - P_s(B_i) \right) &\leq \liminf_{n \to \infty}\widehat{\capE}^{\lambda_n}_{s,g}(F_n) -  \sum_{p=1}^{M} P_s(B_i) \nonumber\\
		&\leq P_s(B_1)  -  \sum_{[i] \in \capI} \left(\frac{|G^i|}{|B_1|}\right)^{\frac{N-s}{N}} P_s(B_1) \nonumber\\
		&\leq P_s(B_1)  - P_s(B_1)  \left(\sum_{[i]\in\capI}\frac{|G^i|}{|B_1|}\right)^{\frac{N-s}{N}}  = 0.
	\end{align}
	From the isoperimetric inequality of $P_s$, we know that $P_s(B_i) \leq P_s(G^i)$ for any $[i] \in \capI$ and the equality holds if and only if $G^i = B_i$ up to translation and negligible sets. Hence, from \eqref{energyMinimalityEachPrticleAsymptotic}, we conclude that the claim holds true.
	
	Next we show that the set $\capI$ of the equivalent classes is actually a finite set. Indeed, we first set $m^{i_p} \coloneqq |G^{i_p}|$ for any $p \in \mathN$ and, since $\sum_{p=1}^{\infty} m^{i_p} = |B_1|$, we can observe that $m^{i_p} \to 0$ as $p \to \infty$ and, moreover, $\mu_{\ell} \coloneqq \sum_{p=\ell+1}^{\infty} m^{i_p} \to 0$ as $\ell \to \infty$. Then, we can choose $\widetilde{p} \in \mathN$ such that $m^{i_{\widetilde{p}}} \geq 2^{-(\widetilde{p}+1)}|B_1|$. Now using the sets $\{G^{i_p}\}_{p=1}^{\infty}$, we construct a new family of sets $\{\widetilde{G}^{i_p}\}_{p=1}^{H}$ for some $H \in \mathN$, depending only on $N$ and $s$, in the following manner; we choose $H \in \mathN$ so large that $H \geq \widetilde{p}$ and set $\widetilde{G}^{i_p} \coloneqq G^{i_p}$ for any $p \in \{1,\cdots ,\,H\}$ with $p \neq \widetilde{p}$ and $\widetilde{G}^{i_{\widetilde{p}}} \coloneqq \eta\,G^{i_{\widetilde{p}}}$ where $\eta^{N} \coloneqq m^{-i_{\widetilde{p}}}(m^{i_{\widetilde{p}}} + \mu_{H})$. Then, we have the volume identity that
	\begin{equation}\label{volumeIdentityReorganizedAsymptotic}
		\sum_{p=1}^{H} \left| \widetilde{G}^{i_p} \right| = \sum_{p=1, \,p\neq\widetilde{p}}^{H} \left| G^{i_p} \right| + \eta^N\, |G^{i_{\widetilde{p}}}| = \sum_{p=1,\,p\neq\widetilde{p}}^{H} m^{i_p} + m^{i_{\widetilde{p}}} + \mu_{H} = |B_1|.
	\end{equation}
	Now we compute the energy for $\{\widetilde{G}^{i_p}\}_{p=1}^{H}$ as follows to show that the total energy of each elements of $\{\widetilde{G}^{i_p}\}_{p=1}^{H}$ is more efficient than that of $\{G^{i_p}\}_{p=1}^{\infty}$; from the definition of $\eta \geq 1$ and Lemma \ref{lemmaScalingEnergy}, we have that
	\begin{align}\label{comparisonNewSetsEnergyAsymptotic}
		\sum_{p=1}^{H} P_s(\widetilde{G}^{i_p}) &= \sum_{p=1,\,p\neq\widetilde{p}}^{H} P_s(G^{i_p}) + \eta^{N-s}\,P_s(G^{i_{\widetilde{p}}}) \nonumber\\
		&= \sum_{p=1}^{\infty} P_s(G^{i_p}) + \left(\eta^{N-s} - 1\right)\,P_s(G^{i_{\widetilde{p}}}) - \sum_{p=H+1}^{\infty}P_s(G^{i_p}) \nonumber\\
		&\leq \sum_{[i] \in \capI} P_s(G^{i_p}) + \frac{2^{\widetilde{p}+1}c_{N,s}\,P_s(B_1)}{|B_1|}\,\mu_{H} - \sum_{p=H+1}^{\infty} P_s(G^{i_p})
	\end{align}
	where $c_{N,s}>0$ is some constant depending only on $N$ and $s$. Here, in the last inequality, we have also used \eqref{keyPropertylowSemicontiVolumeEqualParticlesAsymptotic}. From the isoperimetric inequality of $P_s$ and \eqref{comparisonNewSetsEnergyAsymptotic}, we further obtain that
	\begin{align}
		\sum_{p=1}^{H} P_s(\widetilde{G}^{i_p}) &\leq \sum_{[i] \in \capI} P_s(G^i) + \frac{2^{\widetilde{p}+1}c_{N,s}\,P_s(B_1)}{|B_1|}\, \mu_{H} - C\sum_{p=H+1}^{\infty} \left(m^{i_p}\right)^{\frac{N-s}{N}}  \nonumber\\
		&\leq \sum_{[i] \in \capI} P_s(G^i) + \frac{2^{\widetilde{p}+1}c_{N,s}\,P_s(B_1)}{|B_1|}\, \mu_{H} - C \left(\sum_{p=H+1}^{\infty} m^{i_p}\right)^{\frac{N-s}{N}} \nonumber\\
		&= \sum_{[i] \in \capI} P_s(G^i) + \frac{2^{\widetilde{p}+1}c_{N,s}\,P_s(B_1)}{|B_1|}\, \mu_{H} - C \left( \mu_{H} \right)^{\frac{N-s}{N}}. \nonumber
	\end{align}
	Taking the number $H$ so large that $H \geq \widetilde{p}$ and 
	\begin{equation}\nonumber
		\frac{2^{\widetilde{p}+1}c_{N,s}\,P_s(B_1)}{|B_1|}\, \mu_{H} - C \left( \mu_{H} \right)^{\frac{N-s}{N}} \leq 0,
	\end{equation}
	then we finally obtain that 
	\begin{equation}\label{reductionFiniteElementsMinimizerAsymptotic}
		\sum_{p=1}^{H} P_s(\widetilde{G}^{i_p}) \leq \sum_{[i] \in \capI} P_s(G^i) \leq \liminf_{n \to \infty}\widehat{\capE}^{\lambda_n}_{s,g}(F_n) = \liminf_{n \to \infty}\left(P_s(F_n) - P_{g_{\lambda_n}}(F_n)  \right).
	\end{equation}
	This completes the proof of the claim that the set $\capI$ is finite.
	
	Finally, we show that there exists one number $i_0 \in \mathN$ such that $|G^i|=0$ for any $[i] \in \capI$ with $i \neq i_0$. From \eqref{keyPropertylowSemicontiVolumeEqualParticlesAsymptotic}, there exists at least one number $p'\in\mathN$ such that $|G^{i_{p'}}| > 0$. Then we claim that, if $q \neq p'$, then it holds that $|G^{i_q}|=0$. Indeed, from the previous claim, we can restrict ourselves to consider a finite number of sets $\{\widetilde{G}^{i_p}\}_{p=1}^{H}$, which satisfies \eqref{reductionFiniteElementsMinimizerAsymptotic} and $\sum_{p=1}^{H}|\widetilde{G}^{i_p}| = |B_1|$, instead of $\{G^i\}_{[i]\in\capI}$. Moreover, we may assume that $H \geq p'$. Since we have shown that the sets $\{\widetilde{G}^{i_p}\}_{p=1}^{H}$ are identified with the balls of the volume $|G^{i_p}|$ whenever $|G^{i_p}|>0$ for $p \in \{1,\cdots,H\}$, we can choose the points $\{z^{i_p}\}_{p=1,\,p \neq p'}^{H}$ such that each set $\widetilde{G}^{i_p} + R\,z^{i_p}$ is far away from the others for large $R>1$. We can thus compute the energy as follows; by translation invariance, we have that
	\begin{align}
		\sum_{p=1}^{H} P_s(\widetilde{G}^{i_p}) &= \sum_{p=1\,p \neq p',q}^{H} P_s(\widetilde{G}^{i_p}) + P_s(\widetilde{G}^{i_{p'}}) + P_s(\widetilde{G}^{i_q}) \nonumber\\
		&= \sum_{p=1\,p \neq p',q}^{H} P_s(\widetilde{G}^{i_p}) + P_s(\widetilde{G}^{i_{p'}}) + P_s(\widetilde{G}^{i_q} + R\,z^{i_q}) \nonumber\\
		&= \sum_{p=1\,p \neq p',q}^{H} P_s(\widetilde{G}^{i_p}) + P_s(\widetilde{G}^{i_{p'}} \cup (\widetilde{G}^{i_q} + R\,z^{i_q})) \nonumber\\
		&\qquad + 2\int_{\widetilde{G}^{i_{p'}}} \int_{\widetilde{G}^{i_q} + R\,z^{i_q}} \frac{dx\,dy}{|x-y|^{N+s}}  \nonumber
	\end{align}
	for any $q \in \{1,\cdots,\,H\}$ with $q \neq p'$ and sufficiently large $R>1$. By repeating the same argument finite times for the rest of the sets $\{\widetilde{G}^{i_p}\}_{p=1,\,p \neq p',q}^{H}$ with sufficiently large $R>1$, we obtain the inequality that
	\begin{align}\label{keyEstimateReductionOneElement02Asymptotic}
		\sum_{p=1}^{H} P_s(\widetilde{G}^{i_p}) &\geq P_s\left( \widetilde{G}^{i_{p'}} \cup \bigcup_{p=1,\,p \neq p'}^{H}\left( \widetilde{G}^{i_p} + R\,z^{i_p} \right) \right) \nonumber\\
		&\qquad + \sum_{p=1,\,p \neq p'}^{H} \int_{\widetilde{G}^{i_{p'}}} \int_{\widetilde{G}^{i_p} + R\,z^{i_p}} \frac{dx\,dy}{|x-y|^{N+s}}.
	\end{align}
	Since $\widetilde{G}^{i_{p'}} \cup \bigcup_{p=1,\,p \neq p'}^{H}\left( \widetilde{G}^{i_p} + R\,z^{i_p} \right)$ are the union of disjoint sets, we have, from \eqref{volumeIdentityReorganizedAsymptotic}, that 
	\begin{equation}\nonumber
		\left| \widetilde{G}^{i_{p'}} \cup \bigcup_{p=1,\,p \neq p'}^{H}\left( \widetilde{G}^{i_p} + R\,z^{i_p} \right)\right| = \sum_{p=1}^{H} |\widetilde{G}^{i_{p}}| = |B_1|.
	\end{equation}
	Thus, from \eqref{reductionFiniteElementsMinimizerAsymptotic}, \eqref{keyEstimateReductionOneElement02Asymptotic}, and the minimality of $F_n$, we obtain
	\begin{align}
		\sum_{p=1,\,p \neq p'}^{H} \int_{\widetilde{G}^{i_{p'}}} \int_{\widetilde{G}^{i_p} + R\,z^{i_p}} \frac{dx\,dy}{|x-y|^{N+s}} + P_s(B_1) &\leq \sum_{p=1,\,p \neq p'}^{H} \int_{\widetilde{G}^{i_{p'}}} \int_{\widetilde{G}^{i_p} + R\,z^{i_p}} \frac{dx\,dy}{|x-y|^{N+s}} \nonumber\\
		&\qquad + P_s\left( \widetilde{G}^{i_{p'}} \cup \bigcup_{p=1,\,p \neq p'}^{H}\left( \widetilde{G}^{i_p} + R\,z^{i_p} \right) \right) \nonumber\\
		&\leq  P_s(B_1) \nonumber
	\end{align}
	and it follows that
	\begin{equation}\nonumber
		\sum_{p=1,\,p \neq p'}^{H} \int_{\widetilde{G}^{i_{p'}}} \int_{\widetilde{G}^{i_p} + R\,z^{i_p}} \frac{dx\,dy}{|x-y|^{N+s}} \leq 0
	\end{equation}
	for large $R>1$. Since each term of the sum is non-negative and $|\widetilde{G}^{i_{p'}}|>0$, we conclude that $|\widetilde{G}^{i_p}| = 0$ for all $p \neq p'$.
	
	Therefore, taking into account all of the above arguments, we may conclude that there exist a set $G' \subset \mathR^N$ and points $\{z'_n\}_{n\in\mathN} \subset \mathR^N$ such that, up to extracting a subsequence, we have 
	\begin{equation}\nonumber
		\chi_{F_n - z'_n} \xrightarrow[n \to \infty]{} \chi_{G'} \quad \text{in $L^1_{loc}$}, \quad |G'|=|B_1|. 
	\end{equation}
	From Brezis-Lieb lemma in \cite{BrLi} and the fact that $|G'| = |B_1|$, we obtain that the convergence
	\begin{equation}\nonumber
		\chi_{F_n - z'_n} \xrightarrow[n \to \infty]{} \chi_{G'} \quad \text{in $L^1_{loc}$}
	\end{equation}
	holds in $L^1$ sense. As a consequence, by applying the $\Gamma$-convergence result for the energy $\widehat{\capE}_{s,g}$ as shown in Proposition \ref{propositionGammaConvergenceNonlocalEnergy}, we obtain that $G'$ is a minimizer of the nonlocal perimeter $P_s$, up to translations, because each element of $\{F_n\}_{n}$ is a minimizer of $\widehat{\capE}_{s,g}$. Thus, from the isoperimetric inequality, we conclude that $G'$ coincides with the unit ball up to negligible sets. Finally, we may repeat the above argument for any subsequence of $\{F_n\}_{n\in\mathN}$ and therefore, we conclude that Theorem \ref{theoremAsympMiniLargeVolume} is valid.
\end{proof}
\begin{remark}
	We mention that we are not able to obtain a better convergence of minimizers for Problem \eqref{minimizationScalingModifiedProbelm} than $L^1$-convergence. In general, once we have the $L^1$-convergence and uniform density estimate for minimizers, we can obtain the Hausdorff (possibly $C^1$) convergence of the boundaries of the minimizers to the boundary of the unit ball (see, for instance, \cite{GoNo}). However, we do not know whether the density estimates of minimizers for Problem \eqref{minimizationScalingModifiedProbelm} are valid, while we have the uniform density estimates of minimizers for Problem \eqref{minimizationGeneralizedFunctional}, as shown in Lemma \ref{lemmaUniformDensity}. We might hope that, under some stronger assumptions on $g$ (for instance, some control of the gradient of $g$), the uniform density estimates of minimizers for Problem \eqref{minimizationScalingModifiedProbelm} could be valid.
\end{remark}

\end{document}